\documentclass[11pt]{amsart}
\usepackage{amsmath, amsthm, amssymb}
\usepackage{enumitem}
\usepackage{graphicx} % Required for inserting images
\usepackage{bbm}

\usepackage{fullpage}
\usepackage{empheq}
\usepackage[many]{tcolorbox}
\usepackage{xcolor}
\usepackage{wasysym}
\usepackage{mathtools}
\usepackage{scalerel}

\usepackage{tikz}
\usetikzlibrary{cd}
\usetikzlibrary{calc}
\usetikzlibrary{decorations,decorations.pathreplacing,decorations.markings,decorations.pathmorphing}
\tikzstyle{snake}=[decorate, decoration={snake, segment length=1mm, amplitude=.5mm}]
\newcommand{\tikzmath}[2][]
{\vcenter{\hbox{\begin{tikzpicture}[#1]#2\end{tikzpicture}}}
}
\newcommand{\roundNbox}[6]{
	\draw[rounded corners=5pt, very thick, #1] ($#2+(-#3,-#3)+(-#4,0)$) rectangle ($#2+(#3,#3)+(#5,0)$);
	\coordinate (ZZa) at ($#2+(-#4,0)$);
	\coordinate (ZZb) at ($#2+(#5,0)$);
	\node at ($1/2*(ZZa)+1/2*(ZZb)$) {#6};
}
\tikzset{super thick/.style={line width=3pt}}
\tikzstyle{far>}=[decoration={markings, mark=at position 0.75 with {\arrow{>}}}, postaction={decorate}]
\tikzstyle{far<}=[decoration={markings, mark=at position 0.75 with {\arrow{<}}}, postaction={decorate}]
\tikzstyle{mid>}=[decoration={markings, mark=at position 0.55 with {\arrow{>}}}, postaction={decorate}]
\tikzstyle{mid<}=[decoration={markings, mark=at position 0.55 with {\arrow{<}}}, postaction={decorate}]
\tikzset{super thick/.style={line width=3pt}}
\tikzstyle{far>}=[decoration={markings, mark=at position 0.75 with {\arrow{>}}}, postaction={decorate}]
\tikzstyle{mid>}=[decoration={markings, mark=at position 0.55 with {\arrow{>}}}, postaction={decorate}]
\tikzstyle{mid<}=[decoration={markings, mark=at position 0.55 with {\arrow{<}}}, postaction={decorate}]
\tikzstyle{knot}=[preaction={super thick, white, draw}]
\tikzstyle{coupon}=[draw, very thick, rectangle, rounded corners=5pt]
\tikzset{Rightarrow/.style={double equal sign distance,>={Implies},->},
triplecd/.style={-,preaction={draw,Rightarrow}},
quadruplecd/.style={preaction={draw,Rightarrow,
shorten >=0pt
},
shorten >=1pt,
-,double,double
distance=0.2pt}}
\tikzset{
    tripleline/.style args={[#1] in [#2] in [#3]}{
        #1,preaction={preaction={draw,#3},draw,#2}
    }
}
\tikzstyle{triple}=[tripleline={[line width=.15mm,black] in
      [line width=.7mm,white] in
      [line width=1mm,black]}] 
\tikzset{
    quadrupleline/.style args={[#1] in [#2] in [#3] in [#4]}{
        #1,preaction={preaction={preaction={draw,#4},draw,#3}, draw,#2}
    }
}
\tikzstyle{quadruple}=[quadrupleline={[line width=.3mm,white] in
      [line width=.6mm,black] in
      [line width=1.2mm,white] in
      [line width=1.5mm,black]}]

\definecolor{violet}{RGB}{148,0,211}
\definecolor{DarkGreen}{RGB}{0,150,0}
\definecolor{rufous}{HTML}{A81C07}
\definecolor{boysenberry}{HTML}{873260}
\definecolor{OliveGreen}{HTML}{6D712E}
\definecolor{yellow}{RGB}{200, 120, 60}

\usepackage[plainpages=false,hypertexnames=false,pdfpagelabels]{hyperref}
\definecolor{medium-blue}{rgb}{0,0,.8}
\hypersetup{colorlinks, linkcolor={purple}, citecolor={medium-blue}, urlcolor={medium-blue}}
\newcommand{\arxiv}[1]{\href{http://arxiv.org/abs/#1}{\tt arXiv:\nolinkurl{#1}}}
\newcommand{\arXiv}[1]{\href{http://arxiv.org/abs/#1}{\tt arXiv:\nolinkurl{#1}}}
\newcommand{\mathscinet}[1]{\href{http://www.ams.org/mathscinet-getitem?mr=#1}{\tt #1}}
\newcommand{\doi}[1]{\href{http://dx.doi.org/#1}{{\tt DOI:#1}}}

\DeclareMathOperator{\coev}{coev}
\DeclareMathOperator{\End}{End}

\DeclareMathOperator{\ev}{ev}
\DeclareMathOperator{\Ext}{Ext}

\DeclareMathOperator{\FPdim}{FPdim}
\DeclareMathOperator{\gd}{gd}
\DeclareMathOperator{\gr}{gr}
\DeclareMathOperator{\Hom}{Hom}
\DeclareMathOperator{\id}{id}

\DeclareMathOperator{\Irr}{Irr}

\DeclareMathOperator{\Tor}{Tor}
\DeclareMathOperator{\Tr}{Tr}

\newcommand{\set}[2]{\left\{#1 \middle| #2\right\}}

\newcommand{\Rep}{\mathsf{Rep}}

\def\semicolon{;}
\def\applytolist#1{
    \expandafter\def\csname multi#1\endcsname##1{
        \def\multiack{##1}\ifx\multiack\semicolon
            \def\next{\relax}
        \else
            \csname #1\endcsname{##1}
            \def\next{\csname multi#1\endcsname}
        \fi
        \next}
    \csname multi#1\endcsname}

\def\calc#1{\expandafter\def\csname c#1\endcsname{{\mathcal #1}}}
\applytolist{calc}QWERTYUIOPLKJHGFDSAZXCVBNM;
\def\bbc#1{\expandafter\def\csname bb#1\endcsname{{\mathbb #1}}}
\applytolist{bbc}QWERTYUIOPLKJHGFDSAZXCVBNM;
\def\bfc#1{\expandafter\def\csname bf#1\endcsname{{\mathbf #1}}}
\applytolist{bfc}QWERTYUIOPLKJHGFDSAZXCVBNM;
\def\sfc#1{\expandafter\def\csname s#1\endcsname{{\sf #1}}}
\applytolist{sfc}QWERTYUIOPLKJHGFDSAZXCVBNM;
\def\fc#1{\expandafter\def\csname f#1\endcsname{{\mathfrak #1}}}
\applytolist{fc}QWERTYUIOPLKJHGFDSAZXCVBNM;
\def\rmc#1{\expandafter\def\csname rm#1\endcsname{{\mathrm #1}}}
\applytolist{rmc}QWERTYUIOPLKJHGFDSAZXCVBNM;

\theoremstyle{plain}
\newtheorem{thm}{Theorem}[section]
\newtheorem*{thm*}{Theorem}
\newtheorem{cor}[thm]{Corollary}
\newtheorem{lem}[thm]{Lemma}
\newtheorem{prop}[thm]{Proposition}
\newtheorem{question}[thm]{Question}
\newtheorem*{question*}{Question}
\newtheorem*{claim*}{Claim}

\theoremstyle{defn}
\newtheorem{defn}[thm]{Defintion}

\newtheorem*{trick*}{Trick}

\newtheorem{rem}[thm]{Remark}
\newtheorem*{rem*}{Remark}

\title{Rigidity of non-negligible objects of moderate growth in braided categories}
\author{Pavel Etingof and David Penneys}
\date{\today}

\begin{document}

\begin{abstract} Let $\Bbbk$ be a field, and let $\mathcal{C}$ be a Cauchy complete $\Bbbk$-linear braided category with finite dimensional morphism spaces and $\operatorname{End}(\mathbbm{1})=\Bbbk$. 
We call an indecomposable object $X$ of $\mathcal C$ {\bf non-negligible} if there exists $Y\in \mathcal{C}$ such that $\mathbbm{1}$ is a direct summand of $Y\otimes X$. 
We prove that every non-negligible object $X\in \mathcal{C}$ such that $\dim\operatorname{End}(X^{\otimes n})<n!$ for some $n$ is automatically rigid. 
In particular, if $\mathcal{C}$ is semisimple of moderate growth and weakly rigid, then $\mathcal{C}$ is rigid. 
As applications, we simplify Huang's proof of rigidity of representation categories of certain vertex operator algebras, and we get that for a finite semisimple monoidal category $\mathcal{C}$, the data of a $\mathcal{C}$-modular functor is equivalent to a modular fusion category structure on $\mathcal{C}$, answering a question of Bakalov and Kirillov. 
Furthermore, we show that if $\mathcal{C}$ is rigid and has moderate growth, then the quantum trace of any nilpotent endomorphism in $\mathcal{C}$ is zero. 
Hence $\mathcal{C}$ admits a semisimplification, which is a semisimple braided tensor category of moderate growth. 
Finally, we discuss rigidity in braided r-categories which are not semisimple, which arise in logarithmic conformal field theory. 
These results allow us to simplify a number of arguments of Kazhdan and Lusztig.
\end{abstract}

\maketitle
\tableofcontents

%%%%%%%%%%%%%%%%%%%%%%%%%%%%%%%%%%%%%%%%%%%%%%%%%%%%%%%%%
%%%%%%%%%%%%%%%%%%%%%%%%%%%%%%%%%%%%%%%%%%%%%%%%%%%%%%%%%
%%%%%%%%%%%%%%%%%%%%%%%%%%%%%%%%%%%%%%%%%%%%%%%%%%%%%%%%%
\section{Introduction}
We assume the reader is familiar with tensor categories; many of the definitions omitted below can be found in \cite{MR3242743}.
Let $\Bbbk$ be a field, and let $\cC$ be a Cauchy complete\footnote{A linear category is called \emph{Cauchy complete} if it admits direct sums and is idempotent/Karoubi complete.}
$\Bbbk$-linear braided monoidal category with finite dimensional morphism spaces and $\End(\mathbbm{1})=\Bbbk$. 
We say that an indecomposable $X\in\cC$ is \emph{non-negligible} if there exists $Y\in\cC$ such that $\mathbbm{1}$ is a direct summand in $Y\otimes X$.\footnote{Clearly, if such a $Y$ exists, it can be chosen indecomposable.} 
We say that $X$ has \emph{moderate growth} if there exists $n\in \Bbb N$ such that $\dim\End(X^{\otimes n})<n!$.\footnote{If $\mathcal{C}$ is rigid, symmetric, and abelian, this is equivalent to the usual definition of moderate growth, saying that the length of $X^{\otimes n}$ grows exponentially with $n$. 
Indeed, if $\dim\End(X^{\otimes n})<n!$, then the natural map $\Bbbk S_n\to \End(X^{\otimes n})$ is not injective. 
So if $\operatorname{char}(\Bbbk)=0$, then by Schur-Weyl duality, there exists a Schur functor which annihilates $X$. Thus by Deligne's theorem \cite{MR1944506}, $X$ generates a super-Tannakian category, hence has moderate growth. 
On the other hand, if $\operatorname{char}(\Bbbk)>0$, this follows from \cite[Prop.~4.7(5),(6)]{MR4564264}.}
Our main result, proved in \S\ref{sec:t1}, is:

\begin{thm}
\label{thm:FactorialGrowth}
Every non-negligible object of moderate growth in $\cC$ is rigid.
\end{thm}

We say $\cC$ has \emph{moderate growth} if all indecomposable $X\in\cC$ have moderate growth. 
Thus we get:

\begin{cor}\label{coro2}
If $\cC$ has moderate growth, then every non-negligible $X\in\cC$ is rigid.
\end{cor}

Recall from \cite{MR3134025} that an \emph{r-category} is a monoidal category $\cC$ such that
for every $X\in\cC$, 
the functor
$Z\mapsto\Hom(Z\otimes X,\mathbbm{1})$ is representable by some object $X^*$, and the functor $*:\cC\to \cC^{\rm mop}$ (monoidal and arrow opposite) is an equivalence.
Note that $r$-categories are particular examples of \emph{Grothendieck-Verdier} categories, also known as \emph{$*$-autonomous} categories \cite{MR550878}.
Up to taking opposites, this is exactly the notion of \emph{weak rigidity} in \cite[Def.~5.3.4]{MR1797619}.\footnote{\label{footnote:WeaklyRigidNeqRCat}As a monoidal category may not be equivalent to its opposite, 
not all r-categories are weakly rigid, and vice versa.
We thank a referee for pointing out that the category of left $D$-modules over the dual numbers $D=\bbC[x]/(x^2)$ with the balanced tensor product over $D$ is an r-category by \cite{MR4886205}, but it is not weakly rigid.
Indeed, if ${}_DM$ is not flat, then
$$
\Hom(\mathbbm{1}, M\otimes - ) = \Hom_D(D, M\otimes_D -) \cong \Hom_{\bbC}(\bbC\to M\otimes_D -)
$$
is not exact and thus not representable. 

However, for a semisimple category, both notions are equivalent, as they are just a property of the fusion ring: there is a self-bijection $*$ of its basis $B$ such that for $a,b\in B$, the multiplicity of $1$ in $ab$ is $\delta_{a,b^*}$. 
Moreover, this property is implied by the fusion ring having duality, i.e., being a based ring as in \cite[Def.~3.1.3]{MR3242743}.
}
Clearly, in a semisimple braided r-category with simple $\mathbbm{1}$, every simple object is non-negligible.\footnote{If $Y=X^*$, then 
the morphism 
$Y\otimes X\to \mathbbm{1}$ 
corresponding to ${\rm id}_{X^*}$ 
is nonzero; hence $\mathbbm{1}$ is a direct summand of $Y\otimes X$.} 
Thus we obtain:

\begin{cor}
\label{cor:rCat}
Every semisimple braided r-category of moderate growth is rigid.\footnote{More generally, this corollary holds in a braided ring category (i.e., not necessarily semisimple, but abelian with biexact tensor product, see \cite[Def.~4.2.3]{MR3242743}) of moderate growth if every object is a quotient of a direct sum of tensor products of non-negligible objects. 
Indeed, this follows from Corollary \ref{coro2} and \cite[Cor.~2.36]{MR4533718}.}
\end{cor}

There are several applications of our results.
The first is a simplification of the proof that the representation categories of certain vertex operator algebras (VOAs) are rigid \cite{MR2175093,MR2387861,MR2468370}.
Indeed, before the proof of rigidity in \cite[\S3]{MR2468370}, 
the main assumptions on the VOA $V$ already imply that that $\Rep(V)$ is a finite semisimple 
%\cite[Thm.~3.2]{MR1794264} (see also \cite{MR1615132}) 
braided r-category, and thus our Corollary \ref{cor:rCat} directly applies.

Second, given a finite split\footnote{Here, \emph{split} means that for simple $X\in \cC$, $\End(X)=\Bbbk$, which is automatic if $\Bbbk$ is algebraically closed.} semisimple category $\cC$, the data of a $\cC$-modular functor is essentially equivalent\footnote{\label{footnote:SquareRoot}One must also choose a square root of the central charge.
We refer the reader to \cite[\S1.3]{1509.06811} for a more detailed discussion.} 
to the data of non-degenerate \emph{weak ribbon structure} on $\cC$ \cite[Def.~5.3.5, Thm.~5.7.10]{MR1797619},
which implies that $\cC$ is a finite split semisimple braided r-category with simple $\mathbbm{1}$. It was left open in \cite[Rem.~5.3.7]{MR1797619} whether every such $\cC$ is a modular fusion category, hinging on whether weak rigidity implies rigidity. 
Since finite semisimple monoidal categories have moderate growth, we answer this question affirmatively:
\begin{cor}
\label{cor:ModularFunctor}
Given a finite split semisimple category $\cC$, the data of a $\cC$-modular functor in the sense of \cite{MR1797619} is essentially equivalent\,\textsuperscript{\textup{\ref{footnote:SquareRoot}}} to a modular fusion category structure on $\cC$.\footnote{While this article was being prepared, Andr\'e Henriques informed us that he can use Huang's argument to prove Corollary \ref{cor:ModularFunctor} using modular functors.}
\end{cor}

A third application concerns semisimplification; here, we assume that $\Bbbk$ is algebraically closed.\footnote{Corollaries \ref{cor:Rigidnon-negligibleTempered} and \ref{semisimp} can be extended to general fields, not necessarily algebraically closed. 
However, in this case, the dimension $\Tr_X(\id_X)$ of a non-negligible object can be zero.
For example, the two dimensional irreducible representation $V\in \Rep_{\bbF_2}(\bbZ/3)$
consisting of functions $\bbZ/3\to \bbF_2$ whose values sum to zero
satisfies $\End(V)\cong \bbF_4$, and the quantum trace equals the usual field-theoretic trace, which is zero on $\bbF_2$ and $1$ on $\bbF_4\setminus \bbF_2$.
} 
For rigid $X\in \cC$, let $\Tr_X$ denote the quantum trace constructed from the Drinfeld isomorphism (see \eqref{eq:Trace} below). We first show that $X$ is non-negligible iff $\Tr_X\ne 0$, which is the traditional definition of non-negligibility. 
In particular, direct sums of negligible objects form a thick tensor ideal $\mathcal I\subset \cC$: if $X\in \cC$ and $Z\in \mathcal I$ then $X\otimes Z\in \mathcal I$. Hence we can define 
the Green ring $K(\cC):={\rm Green}(\cC/\mathcal I)$ whose $\Bbb Z_+$-basis is formed by the isomorphism classes of non-negligible objects of $\cC$ and multiplication is defined by the tensor product. 

Next, the proof of Theorem \ref{thm:FactorialGrowth} can be adapted (see \S\ref{c5}) to establish the following characterization of the nilradical $R(X)\subset \End(X)$ for non-negligible $X\in\cC$ of moderate growth:\footnote{The result \cite[Thm.~10.10]{2012.15703} implies our Corollary \ref{cor:Rigidnon-negligibleTempered} for symmetric categories in characteristic zero, subject to some additional moderate conditions.}
\begin{cor}
\label{cor:Rigidnon-negligibleTempered}
If $X$ is non-negligible of moderate growth (hence rigid by Theorem \ref{thm:FactorialGrowth}), 
then $R(X)=\ker(\Tr_X)$. In other words, the quantum trace of 
every nilpotent endomorphism 
of $X$ is zero, but $\Tr_X(\id_X)\ne 0$.
\end{cor}

We also have the following standard lemma, proved in \S\ref{stlempf}.

\begin{lem}\label{stlem} 
If $\cC$ is rigid and the quantum trace of a nilpotent endomorphism of 
each non-negligible object of $\cC$ is zero, then 
this is so for any object of $\cC$.
\end{lem} 

Thus, if $\mathcal{C}$ has moderate growth, then Corollary \ref{cor:Rigidnon-negligibleTempered} enables the semisimplification procedure of \cite[Thm.~2.6]{MR4486913}\footnote{More precisely, the semisimplification procedure of \cite[Thm.~2.6]{MR4486913}
is given in the case of pivotal categories, so works verbatim 
to establish Corollary \ref{semisimp} when $\cC$ is pivotal. 
In the absence of a pivotal structure, we should use a straightforward generalization of \cite[Thm.~2.6]{MR4486913}, which applies to rigid braided categories in which the quantum trace of any nilpotent endomorphism is zero. 
} 
to obtain a semisimple braided tensor category $\overline{\cC}$ of moderate growth. 
We get: 

\begin{cor}\label{semisimp} If $\cC$ has moderate growth, 
then $\cC$ admits a semisimplification $\overline{\cC}$, a semisimple braided tensor category of moderate growth whose Grothendieck ring is $K(\cC)$.
\end{cor}

In particular, applying the main results of \cite{MR1944506} in characteristic zero and of \cite{MR4564264} in positive characteristic, we obtain:

\begin{cor} If $\cC$ is symmetric, then the category $\overline{\cC}$ from Corollary \ref{semisimp} admits a fiber functor to $\operatorname{SuperVect}$ in characteristic zero and to the Verlinde category $\operatorname{Ver}_p$ in positive characteristic, and thus is equivalent to the representation category 
of a linearly reductive affine group scheme in $\operatorname{SuperVect}$, respectively $\operatorname{Ver}_p$. 
 \end{cor}

In characteristic zero, such semisimple symmetric categories are just representation categories of pro-reductive groups, and for characteristics $2,3$, they are representation categories of linearly reductive group schemes classified by Nagata's theorem, see \cite[Section 8]{MR4564264} 
(up to a super-twist of the symmetric structure outside characteristic $2$). 
In positive characteristics $p\ge 5$, there is a conjectural classification 
of them \cite[Conjecture 4.1]{MR4761534}. 
Thus the $\bbZ_+$-ring spanned by non-negligible summands in $X^{*\otimes m}\otimes X^{\otimes n}$, where $X$ is a non-negligible (hence rigid) object of any symmetric category $\cC$ of moderate growth, is very strongly constrained. 

As a fourth application, in \S\ref{sec:NonsemisimpleBraidedRCategories}, we apply our results to braided $r$-categories which are not necessarily semisimple.
Such categories arise in logarithmic conformal field theory.
The results in this section allow us to significantly simplify a number of arguments in \cite{MR1239507}.  

%%%%%%%%%%%%%%%%%%%%%%%%%%%%%%%%%%%%%%%%%%%%%%%%%%%%%
\section{Proofs}

\subsection{Auxiliary lemmas} 
Let $\cC$ be as in the last section.
By convention, we suppress all associators and unitors.
Let $X\in\cC$ be non-negligible, i.e., there is an indecomposable $Y\in\cC$ such that $\mathbbm{1}$ is a direct summand of $Y\otimes X\cong X\otimes Y$. We denote $X$ by an upwards oriented strand and $Y$ by a downwards oriented strand.
We begin with a well-known observation.

\begin{lem}
\label{lem:zig-zag}
If
$\tikzmath{
\draw[mid<] (0,0) arc(-180:0:.3cm);
}
\in \Hom(\mathbbm{1}, X\otimes Y )$
and
$\tikzmath{
\draw[mid<] (0,0) arc(180:0:.3cm);
}
\in \Hom(Y\otimes X, \mathbbm{1})$ are
such that
$z:=
\tikzmath{
\draw[mid>] (.6,-.6) -- (.6,0);
\draw[mid>] (.6,0) arc (0:180:.3cm); \draw[mid>] (0,0) arc(0:-180:.3cm);
\draw[mid>] (-.6,0) -- (-.6,.6);
}
$ 
is invertible, 
then $X,Y$ are rigid with ${}^*X\cong Y\cong X^*$.
\end{lem}
\begin{proof}
By rescaling, we may assume $z=\id_{X}$.
Then 
$
z':=
\tikzmath{
\draw[mid>] (.6,.6) -- (.6,0);
\draw[mid>] (.6,0) arc (0:-180:.3cm); \draw[mid>] (0,0) arc(0:180:.3cm);
\draw[mid>] (-.6,0) -- (-.6,-.6);
}
$
is a sub-diagram of 
$z^2=z=\id_X$, which implies
$z'\neq 0$.
Since $z'^2=z'$, $z'=\id_Y$ as $Y$ is indecomposable.
Using the braiding,
\[
\tikzmath{
\draw[mid>] (0,-.6) -- (0,.6);
}
=
\tikzmath{
\draw[mid>] (.6,-.6) -- (.6,0);
\draw[mid>] (.6,0) arc (0:180:.3cm); \draw[mid>] (0,0) arc(0:-180:.3cm);
\draw[mid>] (-.6,0) -- (-.6,.6);
}
=
\tikzmath{
\draw[mid>] (0,0) arc(0:180:.3cm);
\draw[mid>] (-.6,-1.8) -- (-.6,-.6);
\draw (-.6,-.6) to[out=90,in=-90] (0,0);
\draw (.6,-1.2) to[out=90,in=-90] (0,-.6);
\draw[knot,<-] (0,-.6) to[out=90,in=-90] (-.6,0);
\draw[knot] (0,-1.2) to[out=90,in=-90] (.6,-.6);
\draw[mid<] (0,-1.2) arc(-180:0:.3cm);
\draw[mid>] (.6,-.6) -- (.6,.6);
}
\qquad\qquad\text{and}\qquad\qquad
\tikzmath{
\draw[mid<] (0,-.6) -- (0,.6);
}
=
\tikzmath{
\draw[mid>] (.6,.6) -- (.6,0);
\draw[mid>] (.6,0) arc (0:-180:.3cm); \draw[mid>] (0,0) arc(0:180:.3cm);
\draw[mid>] (-.6,0) -- (-.6,-.6);
}
=
\tikzmath{
\draw[mid<] (0,0) arc(180:0:.3cm);
\draw[mid<] (.6,-1.8) -- (.6,-.6);
\draw (.6,-.6) to[out=90,in=-90] (0,0);
\draw[->] (-.6,-1.2) to[out=90,in=-90] (0,-.6);
\draw[knot] (0,-.6) to[out=90,in=-90] (.6,0);
\draw[knot] (0,-1.2) to[out=90,in=-90] (-.6,-.6);
\draw[mid>] (0,-1.2) arc(0:-180:.3cm);
\draw[mid<] (-.6,-.6) -- (-.6,.6);
}
\,.
\]
Thus $X$ is rigid with ${}^*X\cong Y\cong X^*$.
\end{proof}

For $W,Z\in\cC$, a morphism $r: W\to Z$ is called a \emph{retraction} or a \emph{split surjection} if it admits a \emph{splitting} $s: Z\to W$ such that $rs = \id_Z$. Clearly, the existence of a retraction $r:W\to Z$ is equivalent to $Z$ being a direct summand of $W$. 
Thus there exists a retraction $X\otimes Y\to \mathbbm{1}$. 
Next, we carefully choose one  via the following lemma.

\begin{lem}
\label{lem:GoodRetract}
Let $Z$ be any object containing $\mathbbm{1}$ as a direct summand, and $N\subset \End(Z)$ a nilpotent subspace, i.e., $N^n=0$ for some $n$. There exists a retraction $r: Z\to \mathbbm{1}$ such that for any $a: \mathbbm{1}\to Z$, we have $rNa=0$. 
\end{lem}
\begin{proof}
Let $k\in\bbN$ be minimal for which there is a retraction $r: Z\to \mathbbm{1}$ such that for any $a: \mathbbm{1}\to Z$, $rN^k a=0$. 
(Since $N^n=0$, such a $k\leq n$ exists.)
Suppose for contradiction that $k\ge 2$.
Then there are $a_0: \mathbbm{1}\to Z$, $x\in N$, and $y\in N^{k-2}$ such that $rxya_0=1$. 
Thus $r':=rx$ is also a retraction, and for every 
$a: \mathbbm{1}\to Z$ we have $r'N^{k-1}a\subseteq rN^ka=0$, contradicting the minimality of $k$. 
Thus $k=1$.  
\end{proof}

Let $R(X)\subset \End(X)$ be the nilradical, and choose a retraction
$r=\tikzmath{
\draw[mid>] (0,0) arc(180:0:.3cm);
}
:
X\otimes Y\to \mathbbm{1}
$ and
splitting 
$\tikzmath{
\draw[mid<] (0,0) arc(-180:0:.3cm);
}
:\mathbbm{1}\to X\otimes Y$
satisfying Lemma \ref{lem:GoodRetract} for $Z=X\otimes Y$ and $N:=R(X)\otimes \id_Y\subset \End(X\otimes Y)$.
We prove the contrapositive of Theorem \ref{thm:FactorialGrowth}, i.e., if $X$ is not rigid, then $\dim \End(X^{\otimes n})\ge n!$ for all $n\in \Bbb N$.
Using the braiding, we get a retraction $Y\otimes X\to \mathbbm{1}$ with a splitting:
\begin{equation*}
%\label{eq:YXcupcap}
\tikzmath{
\draw[mid<] (0,0) arc(180:0:.3cm);
}
:=  
\tikzmath{
\draw (0,0) to[out=90,in=-90] (.6,.6);
\draw[knot] (.6,0) to[out=90,in=-90] (0,.6);
\draw[mid>] (0,.6) arc(180:0:.3cm);
}
:Y\otimes X\to \mathbbm{1}
\qquad\qquad\text{and}\qquad\qquad
\tikzmath{
\draw[mid>] (0,0) arc(-180:0:.3cm);
}
:=  
\tikzmath{
\draw (.6,0) to[out=90,in=-90] (0,.6);
\draw[knot] (0,0) to[out=90,in=-90] (.6,.6);
\draw[mid<] (0,0) arc(-180:0:.3cm);
}
:\mathbbm{1}\to Y\otimes X\,.
\end{equation*}
Observe that if any of the following morphisms is invertible, then after redefining 
$\tikzmath{
\draw[mid<] (0,0) arc(180:0:.3cm);
}
:Y\otimes X\to \mathbbm{1}$, 
$X$ is rigid with ${}^*X\cong Y\cong X^*$ by Lemma \ref{lem:zig-zag}.
\begin{equation}
\label{eq:zig-zags}
z:=
\tikzmath{
\draw[mid>] (.6,-.6) -- (.6,0);
\draw[mid>] (.6,0) arc (0:180:.3cm); \draw[mid>] (0,0) arc(0:-180:.3cm);
\draw[mid>] (-.6,0) -- (-.6,.6);
}
\qquad\qquad
\gamma_+:=
\tikzmath{
\draw (.3,-.3) .. controls ++(180:.7cm) and ++(90:.3cm) .. (-.3,.6);
\draw[->] (.6,0) arc(0:-90:.3cm);
\draw[knot,->] (-.3,-.6) .. controls ++(90:.3cm) and ++(180:.5cm) .. (.3,.3);
\draw (.3,.3) arc(90:0:.3cm);
}
=
\tikzmath{
\draw[mid>] (0,0) arc(180:0:.3cm);
\draw[mid>] (.6,-1.2) -- (.6,-.6);
\draw (0,-.6) to[out=90,in=-90] (.6,0);
\draw[knot] (.6,-.6) to[out=90,in=-90] (0,0);
\draw[mid>] (0,-.6) arc(0:-180:.3cm);
\draw[mid>] (-.6,-.6) -- (-.6,.6);
}
\qquad\qquad
\gamma_-:=
\tikzmath{
\draw[->] (-.3,-.6) .. controls ++(90:.3cm) and ++(180:.5cm) .. (.3,.3);
\draw (.3,.3) arc(90:0:.3cm);
\draw[knot] (.3,-.3) .. controls ++(180:.7cm) and ++(90:.3cm) .. (-.3,.6);
\draw[->] (.6,0) arc(0:-90:.3cm);
}
=
\tikzmath{
\draw[mid>] (0,0) arc(180:0:.3cm);
\draw[mid>] (.6,-1.2) -- (.6,-.6);
\draw (.6,-.6) to[out=90,in=-90] (0,0);
\draw[knot] (0,-.6) to[out=90,in=-90] (.6,0);
\draw[mid>] (0,-.6) arc(0:-180:.3cm);
\draw[mid>] (-.6,-.6) -- (-.6,.6);
}
\end{equation}
Recall that since $X$ is indecomposable, the finite dimensional algebra $\End(X)$ is local, hence every element of $\End(X)$ is either invertible or nilpotent \cite[Cor.~4.8(b)]{MR2197389}.
%\footnote{Since $\End(X)$ is finite dimensional, it is Artinian, so $R(X)$ is nilpotent and the quotient $\End(X)/R(X)$ is finite dimensional and semisimple. 
%Since $\End(X)$ has no non-trivial idempotents, $\End(X)/R(X)$ must be a division algebra.
%If $a\in \End(X)$ projects to a nonzero element $a_*\in \End(X)/R(X)$, then $a$ is invertible.
%Indeed, let $b\in \End(X)$ be a lift of $a_*^{-1}$ so that $ab=1-m$ with $m\in R(X)$.
%Then $b(1+m+m^2+\cdots)$ is the inverse of $a$.
%}
We may thus assume that the morphisms in \eqref{eq:zig-zags} are all nilpotent, i.e., lie in $R(X)$.

Now consider the $n$-strand braid group
\[
B_n:=
\left\langle
\sigma_1,\dots, \sigma_{n-1}
\middle|
\begin{aligned}
\sigma_i \sigma_{i+ 1} \sigma_i 
&= 
\sigma_{i+ 1} \sigma_i \sigma_{i+ 1}
&&
\forall\, i
\\
\sigma_i \sigma_j
&=
\sigma_j\sigma_i
&&
j-i>1
\end{aligned}
\right\rangle
\qquad\qquad\qquad
\sigma_i
=
\begin{tikzpicture}[baseline=-.1cm]
\draw (0,-.5) -- (0,.5);
\node at (.35,0) {$\cdots$};
\draw (.6,-.5) -- (.6,.5);
\draw (1.2,-.5) to[out=90,in=-90] (.9,.5);
\draw[knot] (.9,-.5) node[below]{$\scriptstyle i$} to[out=90,in=-90] (1.2,.5);
\draw (1.5,-.5) -- (1.5,.5);
\node at (1.85,0) {$\cdots$};
\draw (2.1,-.5) node[below]{$\scriptstyle n$} -- (2.1,.5);
\end{tikzpicture}\,.
\]
Given an element of $B_n$, we number its strands from left to right on the bottom;
for example, $\sigma_i$ swaps the $i$-th and $i+1$-th strands.
The standard set embedding of $S_n$ into $B_n$ is given by taking a reduced word $w\in S_n$ and writing the same word in $B_n$.
In braid diagrams, these are the elements that can be written in the following form:
\begin{itemize}
\item 
any two distinct strands cross at most once, and
\item
if the $i$-th and $j$-th strands cross with $i<j$, then the $i$-th strand passes \emph{over} the $j$-th strand.
\end{itemize}
Every element of $S_n\subset B_n$ can be written uniquely as
$w_{n-1}w_{n-2}\cdots w_2w_1$ where for each index $1\leq j\leq n-1$, denoting multiple strands by thick colored strands,
\begin{equation}
\label{eq:WordsForSn}
w_j \in \{1,\sigma_{j}, \sigma_{j-1}\sigma_{j},\cdots, \sigma_{1}\cdots \sigma_{j-1}\sigma_{j}\}
=
\left\{\,\,
\tikzmath{
\draw[very thick, red] (-.6,-.3) -- (-.6,.3);
\draw (-.3,-.3) -- (-.3,.3);
\draw (0,-.3) node[below]{$\scriptstyle j{+}1$} -- (0,.3);
\draw[very thick] (.3,-.3) -- (.3,.3);
\node at (-.15,.5) {$\scriptstyle \id$};
}
\,\,,\,\,
\tikzmath{
\draw[very thick, red] (-.6,-.3) -- (-.6,.3);
\draw (0,-.3) node[below]{$\scriptstyle j{+}1$} to[out=90,in=-90] (-.3,.3);
\draw[knot] (-.3,-.3) to[out=90,in=-90] (0,.3);
\draw[very thick] (.3,-.3) -- (.3,.3);
\node at (-.15,.5) {$\scriptstyle \sigma_j$};
}
\,\,,\,\,
\tikzmath{
\draw[very thick, blue] (-.9,-.3) -- (-.9,.3);
\draw (0,-.3) node[below]{$\scriptstyle j{+}1$} to[out=90,in=-90] (-.6,.3);
\draw[knot] (-.6,-.3) to[out=90,in=-90] (-.3,.3);
\draw[knot] (-.3,-.3) to[out=90,in=-90] (0,.3) node[above]{$\scriptstyle \phantom 1$};
\draw[very thick] (.3,-.3) -- (.3,.3);
\node at (-.3,.5) {$\scriptstyle \sigma_{j-1}\sigma_j$};
}
\,\,,\,\,
\dots
\,\,,\,\,
\tikzmath{
\draw (0,-.3) node[below]{$\scriptstyle j{+}1$} to[out=90,in=-90] (-1.1,.3);
\draw[knot] (-1.1,-.3) node[below]{$\scriptstyle 1$} to[out=90,in=-90] (-.8,.3);
\node at (-.8,-.15) {$\scriptstyle\cdots$};
\node at (-.65,.15) {$\scriptstyle\cdots$};
\draw[knot] (-.6,-.3) to[out=90,in=-90] (-.3,.3);
\draw[knot] (-.3,-.3) to[out=90,in=-90] (0,.3) node[above]{$\scriptstyle \phantom 1$};
\draw[very thick] (.3,-.3) -- (.3,.3);
\node at (-.4,.5) {$\scriptstyle \sigma_1\cdots\,\sigma_{j-1}\sigma_j$};
}
\right\}.
\end{equation}
We identify $B_n,S_n$ with their respective images in $B_{n+1},S_{n+1}$ under adding a strand to the right.

\begin{lem}
\label{lem:s!=t}
Suppose $s,t\in S_n$ are distinct.
There are $u,v\in B_{n-1}\subset B_n$ such that $s^{-1}t = u\sigma_{n-1}^{j} v$ in $B_n$, where $j\in \{-1,0,+1\}$.
(However, $u,v\notin S_{n-1}$ necessarily.)
\end{lem}
\begin{proof}
If $s^{-1}t\in B_{n-1}$, we are finished.
Otherwise, since the $n$-th strands of $s,t$ pass behind all other strands, there are $s',t'\in S_{n-1}$ such that
$s= \sigma_i \sigma_{i+1}\cdots\sigma_{n-1}s'$
and
$t= \sigma_j \sigma_{j+1}\cdots\sigma_{n-1}t'$
with $i\neq j$.
Without loss of generality, we may assume 
$s=\sigma_i \sigma_{i+1}\cdots\sigma_{n-1}$
and
$t=\sigma_j \sigma_{j+1}\cdots\sigma_{n-1}$
with $i<j$.
The result now follows by inspection;
representing multiple strands by thick lines, we see
\begin{align*}
s^{-1}t
&=
\sigma_{n-1}^{-1}\cdots
\sigma_i^{-1}
\sigma_j\cdots\sigma_{n-1}
=
\tikzmath{
\draw[very thick, black] (-.6,-1) -- (-.6,1);
\draw (-.3,-1) node[below]{$\scriptstyle i$} -- (-.3,0) to[out=90,in=-90] (.6,1) node[above]{$\scriptstyle n$};
\draw[very thick, knot, red] (0,-1) -- (0,0) to[out=90,in=-90] (-.3,1);
\draw[knot] (.6,-1) node[below]{$\scriptstyle n$} to[out=90,in=-90] (.3,0) node[left, yshift=-.2cm, xshift=.1cm]{$\scriptstyle j$} to[out=90,in=-90] (0,1) node[above]{$\scriptstyle j-1$};
\draw[very thick, knot, blue] (.3,-1) to[out=90,in=-90] (.6,0) to[out=90,in=-90] (.3,1);
\draw[dashed, very thin] (-.8,0) -- (.8,0);
}
=
\tikzmath{
\draw[very thick, black] (-.6,-1) -- (-.6,1);
\draw (-.3,-1) node[below]{$\scriptstyle i$} to[out=90,in=-90] (.3,-.3) to[out=90,in=-90] (.6,.3) to[out=90,in=-90] (.6,1) node[above]{$\scriptstyle n$};
\draw[very thick, knot, red] (0,-1) to[out=90,in=-90] (-.3,-.3) to[out=90,in=-90] (-.3,1);
\draw[knot] (.6,-1) node[below]{$\scriptstyle n$} to[out=90,in=-90] (.6,-.3) to[out=90,in=-90] (.3,.3) to[out=90,in=-90] (0,1) node[above]{$\scriptstyle j-1$};
\draw[very thick, knot, blue] (.3,-1) to[out=90,in=-90] (0,-.3) to[out=90,in=-90] (0,.3) to[out=90,in=-90] (.3,1);
\draw[dashed, very thin] (-.8,-.3) -- (.8,-.3);
\draw[dashed, very thin] (-.8,.3) -- (.8,.3);
}
=
\underbrace{\sigma_{j-1}
\cdots 
\sigma_{n-2}}_{=:u\in B_{n-1}}
\sigma_{n-1}^{-1}
\underbrace{\sigma_{n-2}^{-1}
\cdots
\sigma_{i}^{-1}}_{=:v\in B_{n-1}}.
\qedhere
\end{align*}
\end{proof}

\subsection{Proof of Theorem \ref{thm:FactorialGrowth}}\label{sec:t1}
As above, we assume the morphisms in \eqref{eq:zig-zags} are nilpotent, i.e., lie in the nilradical $R(X)\subset \End(X)$.
For $t\in S_n$, let $f_t\in \End(X^{\otimes n})$ denote the corresponding morphism.
We claim that the set $\set{f_t}{t\in S_n}$ is linearly independent, which implies the result.

Consider the bilinear form
$\Phi: \End(X^{\otimes n})\times \End(X^{\otimes n}) \to \End(\mathbbm{1})=\Bbbk$
given by
\[
\Phi(x,y):=
\tikzmath{
\draw[mid>] (-.5,1.3) arc (180:0:1.3cm and .8cm);
\draw[mid>] (2.1,1.3) -- (2.1,-.3);
\draw[mid>] (2.1,-.3) arc (0:-180:1.3cm and .8cm);
\draw[mid>] (.5,1.3) arc (180:0:.3cm);
\draw[mid>] (1.1,1.3) -- (1.1,-.3);
\draw[mid>] (1.1,-.3) arc (0:-180:.3cm);
\draw[mid>] (.5,.3) -- (.5,.7);
\draw[mid>] (-.5,.3) -- (-.5,.7);
\roundNbox{}{(0,0)}{.3}{.4}{.4}{$y$}
\roundNbox{}{(0,1)}{.3}{.4}{.4}{$x$}
\node at (.05,.5) {$\cdots$};
\node at (.05,-.5) {$\cdots$};
\node at (.05,1.5) {$\cdots$};
\node at (1.65,.5) {$\cdots$};
}\,.
\]
For $s,t\in S_n$, we claim that $ \Phi(f_s^{-1},f_t) = \delta_{s,t}$, which immediately implies that $\set{f_t}{t\in S_n}$ is linearly independent:
\[
0
=
\sum_{t\in S_n} \lambda_t f_t
\qquad\Longrightarrow\qquad
0
=
\Phi\left( f_s^{-1}, \sum_{t\in S_n} \lambda_t f_t \right )
=
\sum_{t\in S_n} \lambda_t \delta_{s,t}
=
\lambda_s
\qquad
\forall s\in S_n.
\]
Clearly $\Phi(f_s^{-1}, f_s) = 1$ by construction.
If $s\neq t$, 
then writing $s=v_{n-1}\cdots v_{1}$ and $t=w_{n-1}\cdots w_{1}$ as in \eqref{eq:WordsForSn},
pick $k$ maximal such that $w_k\neq v_k$.
Then $s^{-1}t$ lies in the image of $S_k$ included into $S_n$,
but the $k$-th lower boundary point
of $s^{-1}t$
does not connect to the $k$-th upper boundary point.
By Lemma \ref{lem:s!=t}, there are $u,v\in B_{k-1}\subset B_n$ such that
$s^{-1}t=u\sigma_{k-1}^{\pm 1}v$.
Then $\Phi( f_s^{-1}, f_t)\in \End(\mathbbm{1})$
has $n-k$ closed loops which can be removed, but the $k$-th strand performs a self-crossing of the form
\[
\gamma_+=
\tikzmath{
\draw (.3,-.3) .. controls ++(180:.7cm) and ++(90:.3cm) .. (-.3,.6);
\draw[->] (.6,0) arc(0:-90:.3cm);
\draw[knot,->] (-.3,-.6) .. controls ++(90:.3cm) and ++(180:.5cm) .. (.3,.3);
\draw (.3,.3) arc(90:0:.3cm);
}
\qquad\qquad
\text{or}
\qquad\qquad
\gamma_-=
\tikzmath{
\draw[->] (-.3,-.6) .. controls ++(90:.3cm) and ++(180:.5cm) .. (.3,.3);
\draw (.3,.3) arc(90:0:.3cm);
\draw[knot] (.3,-.3) .. controls ++(180:.7cm) and ++(90:.3cm) .. (-.3,.6);
\draw[->] (.6,0) arc(0:-90:.3cm);
}
\,,
\]
which lie in $R(X)$ by assumption.
By naturality of the braiding, there is an $a: \mathbbm{1}\to X\otimes Y$ such that 
\[
\pushQED{\qed} 
\Phi(f_s^{-1},f_t) 
=
\tikzmath{
\draw[mid>] (-.3,.8) arc(180:0:.3cm);
\draw[mid>] (-.3,-.2) -- (-.3,.2);
\draw[mid<] (.3,-.2) -- (.3,.8);
\roundNbox{}{(0,-.5)}{.3}{.3}{.3}{$a$}
\roundNbox{}{(-.3,.5)}{.3}{0}{0}{$\gamma_\pm$}
}
\in 
rNa
\underset{\text{(Lem.~\ref{lem:GoodRetract})}}{=}
0.
\qedhere
\popQED
\]

\begin{rem} If $\cC$ is semisimple then $R(X)=0$, so 
Lemma~\ref{lem:GoodRetract} is not needed and the proof simplifies. 
\end{rem} 

\subsection{Non-negligible rigid objects} Let $\Bbbk$ be algebraically closed. 
It is well-known that when $\cC$ is braided and $X\in\cC$ is rigid, we 
have a distinguished quantum trace $\Tr_X: \End(X)\to \End(\mathbbm{1})\cong\Bbbk$
using the Drinfeld isomorphism $u_X: X\to X^{**}$:
\begin{equation}
\label{eq:Trace}
\Tr_X(f)
:=
\tikzmath{
\draw[mid>] (0,-.2) --node[left]{$\scriptstyle X$} (0,.2);
\draw[mid>] (0,.8) node[left, yshift=.2cm]{$\scriptstyle X^{**}$} arc (180:0:.3cm);
\draw[mid>] (.6,.8) --node[right]{$\scriptstyle X^*$} (.6,-.8);
\draw[mid>] (.6,-.8) arc (0:-180:.3cm) node[left, yshift=-.2cm]{$\scriptstyle X$};
\roundNbox{}{(0,.5)}{.3}{0}{0}{$u_X$}
\roundNbox{}{(0,-.5)}{.3}{0}{0}{$f$}
}
=
\tikzmath{
\draw (0,0) to[out=90,in=-90] (.6,.6);
\draw[knot] (.6,0) to[out=90,in=-90] (0,.6);
\draw[mid<] (0,0) arc(-180:0:.3cm);
\draw[mid<] (0,.6) --node[left]{$\scriptstyle X^{*}$} (0,1.2);
\draw[mid>] (.6,1.2) node[right, yshift=.2cm]{$\scriptstyle X$} arc(0:180:.3cm);
\roundNbox{}{(.6,.9)}{.3}{0}{0}{$f$}
}
\qquad\qquad
\text{where}
\qquad\qquad
u_X:=
\tikzmath{
\draw[->] (.3,-.6) node[right]{$\scriptstyle X$} .. controls ++(90:.3cm) and ++(0:.5cm) .. (-.3,.3);
\draw[knot] (-.3,-.3) .. controls ++(0:.7cm) and ++(90:.3cm) .. (.3,.6) node[right]{$\scriptstyle X^{**}$};
\draw[->] (-.6,0) node[left]{$\scriptstyle X^{*}$} arc(-180:-90:.3cm);
\draw (-.3,.3) arc(90:180:.3cm);
}\,.
\end{equation}
It is straightforward to verify that $\Tr_X(fg)=\Tr_X(gf)$ for all $f,g\in\End(X)$.

\begin{lem}
\label{lem:TraceNonzero} 
\mbox{}
\begin{enumerate}[label=(\roman*)]
\item
If X is rigid, then X is non-negligible if and only if there is an $f\in\End(X)$ such that $\Tr_X(f)$ is nonzero. In this case one may take $Y=X^*$ for the witness of non-negligibility of $X$.
\item
If $X$ is rigid indecomposable and ${\rm ker}(\Tr_X)=R(X)$ (i.e., the trace of a nilpotent endomorphism of $X$ is zero), and if $Y$ is indecomposable 
such that $Y\otimes X$ contains $\mathbbm{1}$ as a direct summand, then $Y\cong X^*$. 
\end{enumerate}
\end{lem}

\begin{proof} (i) 
If such an $f$ exists, 
we can take $Y=X^*$, 
$r:= \ev_{X^{*}} (u_X f \otimes \id_{X^*})$
and
$s:= \Tr_X(f)^{-1} \coev_X$
so that $rs=1$, and thus $X$ is non-negligible.
Conversely, if $r: X\otimes Y\to \mathbbm{1}$ is a retraction with splitting $s: \mathbbm{1}\to X\otimes Y$,
then set
\begin{equation}
\label{eq:TraceOne}
f:=
\tikzmath{
\draw[mid>] (.3,.3) -- (.3,-.3);
\draw[mid<] (-.3,.3) to[out=-90,in=90] (-1,-.7) -- (-1,-1.2);
\draw[knot, mid>] (-.3,-.3) to[out=90,in=-90] (-1,.7) -- (-1,1.2);
\roundNbox{}{(0,.6)}{.3}{.2}{.2}{$r$}
\roundNbox{}{(0,-.6)}{.3}{.2}{.2}{$s$}
}
\qquad
\Longrightarrow
\qquad
\Tr_X(f)
\underset{\eqref{eq:Trace}}{=}
\tikzmath{
\draw[mid>] (.3,.3) -- (.3,-.3);
\draw[decoration={markings, mark=at position 0.45 with {\arrow{<}}}, postaction={decorate}] (-.3,.3) to[out=-90,in=90] (-1.6,-.9);
\draw[mid>] (-1,.9) arc(0:180:.3cm);
\draw[mid>, knot] (-1.6,.9) to[out=-90,in=90] (-1,-.9);
\draw[mid>] (-1.6,-.9) arc(-180:0:.3cm);
\draw[knot, mid>] (-.3,-.3) to[out=90,in=-90] (-1,.7) -- (-1,.9);
\roundNbox{}{(0,.6)}{.3}{.2}{.2}{$r$}
\roundNbox{}{(0,-.6)}{.3}{.2}{.2}{$s$}
}
=
\tikzmath{
\draw[mid>] (.3,.3) -- (.3,-.3);
\draw[mid<] (-.3,.3) -- (-.3,-.3);
\roundNbox{}{(0,.6)}{.3}{.2}{.2}{$r$}
\roundNbox{}{(0,-.6)}{.3}{.2}{.2}{$s$}
}
=1.
\end{equation}
(ii) 
If $r: X\otimes Y\to \mathbbm{1}$ is a retraction with splitting $s: \mathbbm{1}\to X\otimes Y$, then 
the $f$ from \eqref{eq:TraceOne} above is invertible as $\Tr_X(f)=1$, $\ker(\Tr_X)=R(X)$, and $X$ is indecomposable.
Dualizing, we see that $f^*: X^*\to X^*$ is an isomorphism.
But $f^*$ can be written as $g\circ h$, where 
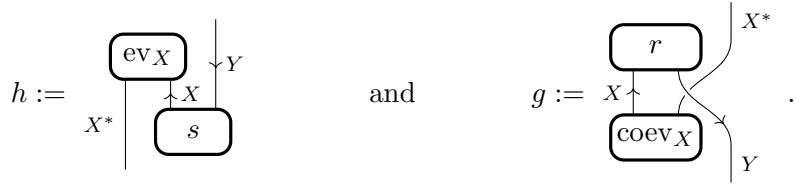
\[
h:=
\tikzmath{
\draw[mid>] (.3,1) --node[right]{$\scriptstyle Y$} (.3,-.2);
\draw (-.9,-1) --node[left]{$\scriptstyle X^*$} (-.9,.2);
\draw[mid>] (-.3,-.2) --node[right]{$\scriptstyle X$} (-.3,.2);
\roundNbox{}{(-.6,.5)}{.3}{.2}{.2}{$\ev_X$}
\roundNbox{}{(0,-.5)}{.3}{.2}{.2}{$s$}
}
\qquad\qquad
\text{and}
\qquad\qquad
g:=
\tikzmath{
\draw (.3,-.3) to[out=90,in=-90] (1,.7) --node[right]{$\scriptstyle X^*$} (1,1.2);
\draw[mid<] (-.3,.3) --node[left]{$\scriptstyle X$} (-.3,-.3);
\draw[knot, mid>] (.3,.3) to[out=-90,in=90] (1,-.7) --node[right]{$\scriptstyle Y$} (1,-1.2);
\roundNbox{}{(0,.6)}{.3}{.3}{.3}{$r$}
\roundNbox{}{(0,-.6)}{.3}{.3}{.3}{$\coev_X$}
}\,.
\]
This implies that $g$ is a retraction with 
splitting $h\circ (f^*)^{-1}$, so $h$ identifies $X^*$ with a direct summand of $Y$. Since $Y$ is indecomposable, it follows that $h$ is an isomorphism. 
\end{proof} 

\subsection{Proof of Corollary \ref{cor:Rigidnon-negligibleTempered}}\label{c5}
When $X$ is rigid and non-negligible, then $\Tr_X: \End(X)\to \Bbbk$ is non-zero by Lemma \ref{lem:TraceNonzero}.
Since $\dim\End(X)<\infty$ and $R(X)$ is a maximal ideal as $X$ is indecomposable,
it suffices to prove that when $X$ is rigid, non-negligible, and $\dim\End(X^{\otimes n})<n!$ for some $n$, then $R(X)\subseteq \ker(\Tr_X)$.
To do so, we show that if $X$ is rigid and non-negligible and there is a nilpotent $f\in R(X)$ with $\Tr_X(f)\neq 0$, then $\dim\End(X^{\otimes n})\ge n!$ for all $n$.

First, we mimic the proof of Lemma \ref{lem:GoodRetract}.
Set $Y=X^*$, $r:= \ev_{X^{*}}(u_X\otimes \id_{X^*})$ and $N:=R(X)\otimes \id_{X^*}$.
Let $k\in\bbN$ be minimal such that $rN^ka= 0$
for all $a: \mathbbm{1}\to X\otimes X^*$.
Since $\Tr_X(f)\neq 0$, taking $a=\coev_X$ shows that $k\geq 2$.
Pick $g\in R(X)^{k-1}$ and $s_0:\mathbbm{1}\to X\otimes X^*$ so that $r(g\otimes \id_{X^*})s_0=1$.
Setting $r_0:= r(g\otimes \id_{X^*})$, we have  $r_0s_0=1$ 
and $r_0Ns_0=0$.
Observe that 
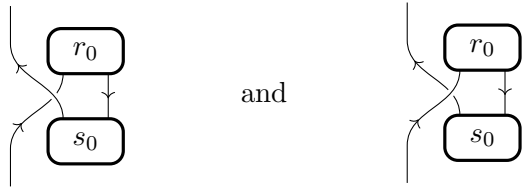
\[
\tikzmath{
\draw[mid>] (.3,.3) -- (.3,-.3);
\draw[mid<] (-.3,.3) to[out=-90,in=90] (-1,-.7) -- (-1,-1.2);
\draw[knot, mid>] (-.3,-.3) to[out=90,in=-90] (-1,.7) -- (-1,1.2);
\roundNbox{}{(0,.6)}{.3}{.2}{.2}{$r_0$}
\roundNbox{}{(0,-.6)}{.3}{.2}{.2}{$s_0$}
}
\qquad\qquad
\text{and}
\qquad\qquad
\tikzmath{
\draw[mid>] (.3,.3) -- (.3,-.3);
\draw[mid>] (-.3,-.3) to[out=90,in=-90] (-1,.7) -- (-1,1.2);
\draw[knot, mid<] (-.3,.3) to[out=-90,in=90] (-1,-.7) -- (-1,-1.2);
\roundNbox{}{(0,.6)}{.3}{.2}{.2}{$r_0$}
\roundNbox{}{(0,-.6)}{.3}{.2}{.2}{$s_0$}
}
\]
both lie in $R(X)$ as they contain $g\in R(X)^{k-1}\subset R(X)$.
Now by mimicking the proof of Theorem \ref{thm:FactorialGrowth},
we can use $r_0,s_0$ in place of the cap/cup to see that $\dim \End(X^{\otimes n})\ge n!$ for all $n$.
\qed

\subsection{Proof of Lemma \ref{stlem}}\label{stlempf} 

Suppose 
$$
Z=\bigoplus_{i=1}^r V_i\otimes X_i,
$$ 
where $X_i$ are pairwise non-isomorphic indecomposable objects of $\cC$ and $V_i$ are finite dimensional vector spaces. Then the nilradical $R(Z)$ of $\End(Z)$ has the form 
$$
R(Z)=\bigoplus_{i=1}^r \End(V_i)\otimes R(X_i) \oplus \bigoplus_{i\ne j}\Hom(V_i,V_j)\otimes \Hom(X_i,X_j).
$$
So $\Tr_Z$ vanishes on $R(Z)$, and on 
the quotient 
$$
\End(Z)/R(Z)\cong \bigoplus_{i=1}^r \End(V_i)
$$
induces the linear functional $T(f_1,\cdots,f_r):=\sum_{i=1}^r c_i\Tr_{\End(V_i)}(f_i)$, 
where $c_i:=\Tr_{X_i}({\rm id}_{X_i})$. 
If $f: Z\to Z$ is a nilpotent endomorphism, then its image $(f_1,\dots, f_r)\in \End(Z)/R(Z)$ satisfies that $f_i\in\End(V_i)$ is nilpotent for each $i$, and thus $\Tr_{\End(V_i)}(f_i)=0$ for each $i$.
We conclude that $\Tr_Z(f)=T(f_1,\cdots,f_r)=0$.
%.
%If $X_i$ is negligible, then $c_i=0$, and if $X_i$ is non-negligible, then $\Tr_{\End(V_i)}(f_i)=0$ as $f_i$ is nilpotent.
%We conclude that $\Tr_Z(f)=T(f_1,\cdots,f_r)=0$. 
\qed

%%%%%%%%%%%%%%%%%%%%%%%%%%%%%%%%%%%%%%%%%%%%%%%%%%%%%
\section{Remarks}

\subsection{The linearized category of crystals} \label{sec:cry}

The assumption that $\cC$ is braided in Theorem 
\ref{thm:FactorialGrowth} cannot be removed. 
A counterexample is the $\Bbb C$-linearization ${\rm Crys}(G)$ of the {\it category of crystals} for the quantum group $G_q$ attached to a simply connected simple complex group $G$ \cite{MR2219257}, which is a semisimple category of moderate growth but is not braided, nor rigid (only weakly rigid). 
This category is a limit when $q\to 0$ of the braided, rigid categories $\Rep(G_q)$ with the same Grothendieck ring, but the braiding and rigidity are destroyed in the limit; in fact, the only rigid objects in ${\rm Crys}(G)$ are multiples of $\mathbbm{1}$. Indeed, in $\Rep(G_q)$, if $V_\lambda$ is the irreducible representation with highest weight $\lambda$ and $f: V_\lambda\to V_\lambda^{**}$ is an isomorphism, then $\Tr(f)\Tr((f^*)^{-1})=(\dim_q V_\lambda)^2$, which goes to infinity as $q\to 0$ for any $\lambda\ne 0$, as do those eigenvalues of the squared braiding on $V_\lambda\otimes V_\mu$ which are negative powers of $q$. 
The category ${\rm Crys}(G)$ is, however, a {\it coboundary category} \cite{MR2219257}: it carries a symmetric \emph{commutor} (a natural isomorphism $\sigma_{X,Y}:X\otimes Y\to Y\otimes X$) satisfying the Reidemeister move R2, but not R3. 
This shows that the move R3 plays a crucial role in the proof of Theorem \ref{thm:FactorialGrowth}.

To make things concrete, let us restrict to the case $G=SL(2)$. 
In this case, ${\rm Crys}(G)$ can be realized 
as the \emph{asymptotic Temperley-Lieb-Jones category} $TLJ(\infty)$.
Namely, recall that
the usual Temperley-Lieb-Jones category $TLJ(\delta)$ is defined 
so that the circle evaluates to $\delta$, while the zig-zag evaluates to 
$1$ (namely, $TLJ(\delta)=\Rep(SL(2)_q)$ where $\delta=-q-q^{-1}$). 
By renormalizing the diagrams, we may arrange that the circle evaluates to $1$, while the zig-zag evaluates to $\delta^{-1}$. 
This allows us to specialize $TLJ(\delta)$ at $\delta=\infty$ (i.e., $\delta^{-1}=0$), which yields a semisimple category $TLJ(\infty)$ with Grothendieck ring the representation ring of $SL(2)$ where the circle evaluates to $1$ 
while the zig-zag evaluates to $0$ (see \cite[\S3]{1308.2402}), which is not braided, nor rigid (although is weakly rigid). Namely, $\Hom([n],[m])$ 
in $TLJ(\infty)$ has the usual basis of Temperley-Lieb-Jones diagrams with $n$ inputs and $m$ outputs and composition given by concatenation of diagrams and removing the circles, but when the concatenation contains a zig-zag, instead of straightening it, the composition is declared to be zero. 

The endomorphism algebra of the object $[n]$ in $TLJ(\infty)$ is thus the asymptotic Temperley-Lieb-Jones algebra $TLJ_n(\infty)$ spanned by diagrams with $n$ inputs and $n$ outputs (no longer generated by the usual generators $e_i$ if $n\ge 3$, however). 
We obtain a quick proof that $TLJ_n(\infty)$ and hence the category $TLJ(\infty)$ are semisimple by 
defining a filtration\footnote{Here, by a \emph{filtration} on an algebra $A$, we mean a descending sequence $A\supset I_1\supset I_2\supset \cdots$ of two-sided ideals.
%For any ideal $I\subset A$ which is also a unital algebra itself, the
%unit $f\in I$ is automatically central in $A$:
%$[a,f]=[a,f^2]=[a,f]f+f[a,f]=2[a,f]$, so $[a,f]=0$.
When $A=TLJ_n(\infty)$, $I_k$ is spanned by the diagrams with at least $k$ cups, and $I_k/I_{k+1}$ is a semisimple unital algebra whose unit $f_k$ is central in $A/I_{k+1}$. 
Thus $A/I_{k+1}= (1-f_k)A/I_{k+1}\oplus f_kA/I_{k+1}=A/I_k\oplus I_k/I_{k+1}$,
and so $A=\oplus_k I_k/I_{k+1}$ is semisimple.
} 
on $TLJ_n(\infty)$ in which the degree of a diagram $u$ is its number of cups.

We can even identify $TLJ_n(\infty)$ as a multi-matrix algebra by observing that the Jones-Wenzl idempotents exist and are non-zero for all $n\in\bbN$.
Indeed, the usual Jones-Wenzl recurrence relation \cite{MR873400} simplifies greatly as zig-zags are zero:
\[
JW_{n+1} = JW_n\otimes 1 - (JW_n\otimes 1) e_n (JW_n\otimes 1).
\]
One can also check via this recurrence relation that $JW_n$ is the linear combination of all projections $p$ with no nested cups/caps where the coefficient of $p$ is $(-1)^{\#\text{ caps in }p}$.
Matrix units for the summands of $TLJ_n(\infty)$ are then obtained by cabling the `waists' of basis elements by the appropriate Jones-Wenzl idempotents.
It then follows that $TLJ_n(\infty)$ is a multi-matrix algebra, as diagrams whose through strings are cabled by different Jones-Wenzl idempotents are orthogonal.\footnote{The article \cite{MR4958085} treats these categories in detail, including the semisimplicity, the formula for $JW_n$ and matrix block decomposition.
}

All these counterexamples have infinitely many simple objects, however. 
This gives rise to the following question, which is open even for weakly rigid categorifications of based rings (with duality).

\begin{question}\label{q1}
Can Corollary \ref{cor:rCat} be generalized to {\bf finite} semisimple weakly rigid categories (not necessarily braided)?\end{question}

\begin{rem}
The answer to Question \ref{q1} is ``yes" when $\cC$ is a semisimple (not necessarily braided) full monoidal subcategory of endomorphisms of an infinite von Neumann factor which is weakly rigid when $*:\cC\to \cC^{\rm mop}$ is taking conjugates (so that $*\circ *\cong \id_\cC$).
Indeed, rigidity follows by \cite[Thm.~4.1]{MR1059320}; see also \cite[Prop.~7.17]{MR3342166} for bimodules over a von Neumann factor.
\end{rem} 

\subsection{Rigidity for small growth dimension in non-braided categories} 
For a positive element $Z$ of a based ring $A$ as in \cite[Def.~3.1.3]{MR3242743}, define the {\it growth dimension} $\gd(Z)$ to be
$$
\gd(Z):=\lim_{n\to \infty}{\rm length}(Z^n)^{\frac{1}{n}}
$$
(see \cite[Def.~4.1]{MR4564264}), where the length of a positive element is the sum of its coefficients.\footnote{This limit exists since the sequence $d_n(Z):={\rm length}(Z^n)$ is super-multiplicative, i.e., $d_{n+m}(Z)\ge d_n(Z)d_m(Z)$, see \cite[\S4]{MR4564264}.} For example, 
for $A$ the character ring of a complex reductive group $G$, 
$\gd(Z)$ is the usual dimension of the representation $Z$, and if $A$ is finite, then $\gd(Z)$ coincides with the Frobenius-Perron dimension $\FPdim(Z)$. 

\begin{prop} \label{lessthan4}
Let $X,Y$ be simple objects of a split semisimple weakly rigid category $\cC$ over $\Bbbk$ such that 
$\mathbbm{1}$ is a direct summand in $Y\otimes X$ (and thus in $X\otimes Y$). If $\gd(Y\otimes X)<4$ then $X$ is rigid and $Y\cong X^*$.
\end{prop} 

\begin{proof} Fix a retraction (cap) $r_+: Y\otimes X\to \mathbbm{1}$
and splitting (cup) $s_+: \mathbbm{1}\to Y\otimes X$ such that $r_+s_+=1$
and similarly $r_-: X\otimes Y\to \mathbbm{1}$
and $s_-: \mathbbm{1}\to X\otimes Y$ such that $r_-s_-=1$.
We therefore have the zig-zags 
$$
z_1=(r_+\otimes 1)(1\otimes s_-),\ 
z_2=(1\otimes r_+)(s_-\otimes 1),\
z_3=(r_-\otimes 1)(1\otimes s_+),\
z_4=(1\otimes r_-)(s_+\otimes 1)\in \Bbbk.
$$
Moreover, $z_1^2=z_1z_2=z_2^2$, so $z_1=z_2=:z$, and similarly 
$z_3=z_4=:z_*$.\footnote{Note that we can rescale $z$ by $\lambda\in \Bbbk^\times$ by rescaling $r_\pm,s_\pm$, but then $z_*$ will multiply by $\lambda^{-1}$, so the product $zz_*$ is defined canonically. If $X$ is rigid and 
$Y=X^*$ then it is easy to show that $zz_*=1/|X|^2$, where $|X|^2$ is the M\"uger squared norm of $X$ (\cite{MR3242743}, Def.~7.21.2).}

For a non-crossing matching $u$ on the word $(YX)^n$, 
let $D(u)$ be the number of $YX$-matchings in $u$ (so the
number of $XY$-matchings is $n-D(u)$). 
Consider the pairing 
$$
\Phi: \Hom((Y\otimes X)^{\otimes n},\mathbbm{1})\times \Hom(\mathbbm{1},(Y\otimes X)^{\otimes n})\to \Bbbk
$$
given by $\Phi(f,g)=f\circ g$. For every non-crossing matching 
$u$ of $[1,2n]$ we have elements $f_u\in \Hom((Y\otimes X)^{\otimes n},\mathbbm{1})$ defined using $r_+,r_-$ and $g_u\in \Hom(\mathbbm{1},(Y\otimes X)^{\otimes n})$ defined using $s_+,s_-$. It is easy to see that for any two non-crossing matchings $u,v$, we have 
$$
\Phi(f_u,g_v)=z^{a(u,v)}z_*^{b(u,v)},
$$
where $a(u,v)-b(u,v)=D(u)-D(v)$, $a(u,v)+b(u,v)>0$ if $u\ne v$, and $\Phi(f_u,g_u)=1$. 
Thus if $z=0$ or $z_*=0$, then the matrix $[\Phi(f_u,g_v)]$ is upper or lower triangular in any ordering compatible with the $D$-grading, 
with $1$ on the diagonal, hence nondegenerate. 
It then follows that $\lbrace g_u\rbrace$ are linearly independent. Since the number 
of non-crossing matchings is $C_n$ (the $n$-th Catalan number), 
this implies that 
$$
\dim \Hom(\mathbbm{1}, (Y\otimes X)^{\otimes n})\ge C_n.
$$
Since $\lim_{n\to \infty}C_n^{\frac{1}{n}}=4$, it follows that 
$\gd(Y\otimes X)\ge 4$. Thus if $\gd(Y\otimes X)<4$ then 
$z,z_*\ne 0$, implying that $X$ is rigid and $Y\cong X^*$. 
\end{proof} 

\begin{rem} Note that the bound in Proposition \ref{lessthan4} is sharp: for the crystal $X\in {\rm Crys}(SL_2)$ of the $2$-dimensional representation, $Y=X$ and $\gd(X)=2$, so $\gd(Y\otimes X)=4$, but $X$ is not rigid.

Moreover, if ${\rm gd}(Y\otimes X)=4$, it is possible that exactly one of $z,z_*$ is zero. 
Namely, we may consider the universal pivotal category defined in \cite[\S3.2]{MR4775671} which in our notation 
would be natural to denote $TLJ(\delta_1,\delta_2)$. 
Its morphisms are oriented Temperley-Lieb-Jones diagrams, with counterclockwise circles evaluated to $\delta_1$ and 
clockwise circles to $\delta_2$. We can renormalize the diagrams 
so that both circles evaluate to $1$ while $z=\delta_1^{-1}$, 
$z_*=\delta_2^{-1}$. 
Now we can take $\delta_1$ or $\delta_2$ to $\infty$ to get categories $TLJ(\infty,\delta_2)$, $TLJ(\delta_1,\infty)$ where $z=0,z_*\ne 0$, respectively $z\ne 0,z_*= 0$.

Note that the category $TLJ(\delta,\infty)$ (and likewise $TLJ(\infty,\delta)$) is semisimple for $\delta\ne 0$. 
Indeed, the endomorphism algebra of a word $W$ in $X,Y$ in $TLJ(\delta,\infty)$ is the {\it asymptotic lopsided Temperley-Lieb-Jones algebra} $TLJ_W(\delta,\infty)$ in which both circles evaluate to $1$, $z=\delta^{-1}$ and $z_*=0$. 
Define a filtration on this algebra in which the degree of a diagram $u$ is the number of $s_+:\mathbbm{1}\to Y\otimes X$ minus the number of $r_+:Y\otimes X\to \mathbbm{1}$ in $u$. 
It is easy to check that this indeed defines a filtration, and 
the associated graded algebra
${\rm gr}(TLJ_W(\delta,\infty))\cong TLJ_W(\infty,\infty)$, where $z=0=z_*$.

Thus, it suffices to show that $TLJ_W(\infty,\infty)$ is semisimple;
we argue similarly to \S\ref{sec:cry}.
Define the degree of a diagram $u$ in $TLJ_W(\infty,\infty)$ as its number of cups.
Then $I_k/I_{k+1}$ is a direct sum of matrix algebras indexed over sub-words of $W$ which can be obtained by successively removing instances of $XY$ or $YX$ from $W$.
(These subwords are the words which are retractions of $W$ via $r_\pm$.)
\end{rem} 

\subsection{Categories of non-moderate growth}

The assumption that $X$ has moderate growth in Theorem 
\ref{thm:FactorialGrowth} cannot be removed, even for symmetric categories. 
A counterexample is the symmetric oriented Brauer category $OB(\infty)$, 
which is the limit as $t\to \infty$ of the oriented Brauer categories 
$OB(t)$, also known as the Deligne categories $\Rep(GL_t)$, $t\in \Bbb C$ \cite[\S9.12]{MR3242743}. 
Namely, recall that $OB(t)$ has objects $[n,m]$ for $n,m\in \bbN$, and $\Hom([n_1,m_1],[n_2,m_2])$ 
is spanned by appropriate walled Brauer diagrams, with composition given by concatenation of diagrams, so that the circle evaluates to $t$. 
Similar to 
the construction of $TLJ(\infty)$ in \S\ref{sec:cry}, we may 
renormalize these diagrams by suitable powers of $t$ so that 
the circle evaluates to $1$, but the zig-zags are $t^{-1}$. 
This allows us to 
evaluate at $t=\infty$ (i.e., $t^{-1}=0$) and, upon Cauchy completion, obtain a symmetric category $OB(\infty)$. 
In this category, diagrams are composed by concatenation and removing circles, but if the concatenated diagram contains a zig-zag, the composition is declared to be zero. 

For example, $\End([n,m])$ is the asymptotic walled Brauer algebra $W_{n,m}(\infty)$ obtained by the above limiting procedure from the usual walled Brauer algebra $W_{n,m}(t)$. This algebra has the usual basis of walled Brauer diagrams with $n$ inputs and $n$ outputs to the left of the wall and $m$ inputs and $m$ outputs to the right of the wall. As in \S\ref{sec:cry}, we can define a filtration on $W_{n,m}(\infty)$ in which the degree of a diagram $u$ is its number of cups, and the associated graded algebra $\gr(W_{n,m}(\infty))=\bigoplus_k I_k/I_{k+1}$ is manifestly semisimple with $k$-th summand
\[
I_k/I_{k+1} \cong \bbC[S_{n-k}]\otimes \bbC[S_{m-k}]\otimes \operatorname{Mat}\left(\frac{n!m!}{(n-k)!(m-k)!k!}\right).
\]
This implies that $W_{n,m}(\infty)$ and thus the category $OB(\infty)$ are semisimple, with simples $X_{\lambda,\mu}$ labeled by pairs of partitions, and the same fusion rules as in $OB(t)$. 
However, the only rigid objects in this category are multiples of $\mathbbm{1}$: the dimension $\dim V_{\lambda,\mu}(t)$ in $OB(t)$ is a non-constant polynomial of $t$, hence goes to infinity as $t\to \infty$. 
Unsurprisingly, these are also the only objects in this category that have moderate growth. 

One can also perform the same limiting procedure to the unoriented Brauer category
$UB(t)=\Rep(O_t)$ (\cite[\S9.12]{MR3242743}), 
getting a semisimple category $UB(\infty)$ which is likewise a counterexample to Theorem \ref{thm:FactorialGrowth} without moderate growth.
 
\subsection{Categories with nilpotent endomorphisms of nonzero trace} 

The moderate growth assumption in Corollary \ref{cor:Rigidnon-negligibleTempered} cannot be removed, even for symmetric categories, as was first observed by Deligne \cite[\S5.8]{MR2348906}. 
Namely, there exist symmetric categories 
(of non-moderate growth) in which the trace of a nilpotent endomorphism of an object can be nonzero. A nice example is the oriented Brauer category $OB(t)$ over a field $\Bbbk$ of characteristic $2$, where $t\in \Bbbk$ is not equal $0$ or $1$. Then we can take $V:=[1,0]$, 
$X=V^{\otimes 2}=[2,0]$ and $z=1-s\in \End(X)$ where $s$ is the swap $V^{\otimes 2}\to V^{\otimes 2}$. Then $z^2=0$ but $\Tr(z)=\Tr(1)-\Tr(s)=t^2-t\ne 0$. Many other examples of categories containing nilpotent endomorphisms with nonzero trace appear in 
\cite{MR4450143}. Of course, such a symmetric category cannot have monoidal functors to abelian symmetric tensor categories, since in the latter 
the trace of any nilpotent endomorphism $z$ must be zero (as $z$ is strictly upper triangular in the filtration by kernels of its powers).

\section{Rigidity in braided $r$-categories} 
\label{sec:NonsemisimpleBraidedRCategories}

This section is inspired by the paper \cite{MR1239507}. 
Its goal is to apply our results 
to braided $r$-categories that are not necessarily semisimple. 
Such categories arise in logarithmic conformal field theory \cite{MR4883790}.
In particular, this allows us to significantly simplify a number of arguments in \cite{MR1239507}.  

\subsection{2-out-of-3 rigidity} 
Let $\cC$ be a braided monoidal category. Recall that if $X\in \cC$ is rigid then so is $X^*\cong {}^*X$. 

Assume in addition that $\cC$ is abelian and artinian with right exact tensor product, so that $\Tor(X,Y)$ is well-defined for every $X,Y\in\cC$, possibly as a pro-object.\footnote{For a finite artinian category (modules over a finite dimensional algebra $A$), a right exact functor is always tensoring with an $A$-bimodule over $A$. 
So for every object $X$, there is an $A$-module $M_X$, and $\Tor(X,Y)$ is the derived tensor product $M_X\otimes_A^L Y$. 
In the infinite case it is similar, but we must consider comodules over a coalgebra $C$ (i.e., topological module over $A=C^*$) and the functor is tensoring with this topological module. 
This will define the $\Tor$-object in the pro-completion of $C$-comod.}
An object $X\in \cC$ is called {\bf flat} if the functor $X\otimes -$ is exact. 

\begin{lem}[{c.f.~\cite[Cor.~1, p.~441]{MR1239507}}]\label{L0}
Every rigid object in $\cC$ is flat.
\end{lem} 
\begin{proof} 
Although this follows formally from the adjunctions $X^* \otimes - \dashv X \otimes - \dashv {}^*X\otimes -$, we provide a direct proof for convenience and simplicity.
We show that if a morphism $f: Y\to Z$ is a monomorphism, then so is the morphism $\id_X\otimes f: X\otimes Y\to X\otimes Z$, i.e., that for every object $T\in \cC$, the associated linear map $f_{T,X}:\Hom(T,X\otimes Y)\to \Hom(T,X\otimes Z)$ is injective.
Since $X$ is rigid, this map can be viewed as the pullback map 
$\widetilde f_{T,X}: \Hom(X^*\otimes T,Y)\to \Hom(X^*\otimes T,Z)$, which is obviously injective. 
\end{proof} 

\begin{lem}  \label{L1} 
Consider a short exact sequence in $\cC$ of the form
$
0\to Y\xrightarrow{i} Z\xrightarrow{p} X\to 0
$.
Then:
\begin{enumerate}[label=(\roman*)]
\item
If $X,Y$ are rigid then $Z$ is rigid. 
\item
If $Z,X$ are rigid then $Y$ is rigid. 
\item
If $Y,Z$ are rigid and the natural morphism $i^*:Z^*\to Y^*$ is an epimorphism, then $X$ is rigid.  
\end{enumerate}
\end{lem} 

\begin{proof} (i)
Since $X,Y$ are rigid, by Lemma \ref{L0} they are flat. Using the long exact sequence of Tor groups, 
we see that $Z$ is flat as well. Also $Z$ represents an element 
$$\alpha\in \Ext^1(X,Y)\cong \Ext^1(\mathbbm{1},Y\otimes X^*)\cong \Ext^1(Y^*,X^*).$$ 
Let $Z^\vee$ be the extension 
of $Y^*$ by $X^*$ defined by the element $-\alpha$, 
so it is also flat. 
Since $X,Y,Z,X^*,Y^*,Z^\vee$ are flat, we see that the object $Z\otimes Z^\vee$ has a 3-step filtration with successive quotients 
$Y\otimes X^*,X\otimes X^*\oplus Y\otimes Y^*,X\otimes Y^*$, meaning that we have subobjects 
$$
F_1=Y\otimes X^*
\subset 
F_2=
\ker(
Z\otimes Z^\vee \to X\otimes Y^*
)
\subset 
F_3=
Z\otimes Z^\vee
$$
with $F_2/F_1\cong X\otimes X^*\oplus Y\otimes Y^*$ and $F_3/F_2\cong X\otimes Y^*$.
By construction of $Z^\vee$,
the morphism
$\coev_X\oplus \coev_Y: \mathbbm{1}\to X\otimes X^*\oplus Y\otimes Y^*$ lifts to a morphism $\mathbbm{1}\to F_2$, hence to a morphism $\mathbbm{1}\to Z\otimes Z^\vee$, which we denote by $\coev_Z$. 
To see this, one looks at the exact sequence
$$
0\to F_1=Y\otimes X^*\to F_2 \to F_2/F_1 \cong X\otimes X^*\oplus Y\otimes Y^*\to 0
$$
and computes that the pullback in $\Ext^1(\mathbbm{1}, Y\otimes X^*)$ of its class in $\Ext^1(X\otimes X^*\oplus Y\otimes Y^*, Y\otimes X^*)$ is trivial since we used $-\alpha$ to define $Z^\vee$, whence our lift $\coev_Z: \mathbbm{1}\to Z\otimes Z^\vee$ exists.
We remark that writing 
$i^\vee: Z^\vee\to Y^*$
and
$p^\vee:X^*\to Z^\vee$,
the following diagram commutes.
$$
\begin{tikzcd}[column sep=4em]
Y\otimes Y^*
\arrow[d,"i\otimes \id_{Y^*}"]
&
\mathbbm{1}
\arrow[d,"\coev_Z"]
\arrow[l,"\coev_Y",swap]
\arrow[r,"\coev_X"]
&
X\otimes X^*
\arrow[d,"\id_X\otimes p^\vee"]
\\
Z\otimes Y^*
&
Z\otimes Z^\vee
\arrow[l,"\id_Z\otimes i^\vee",swap]
\arrow[r,"p\otimes \id_{Z^\vee}"]
&
X\otimes Z^\vee
\end{tikzcd}
$$
This can be seen by analyzing the other faces of the following larger diagram.
$$
\begin{tikzcd}
\mathbbm{1} 
\arrow[rr,"\coev_X"]
\arrow[dd,"\coev_Y"]
\arrow[dr,"\coev_Z"]
&& 
X\otimes X^* 
\arrow[dd,hook]
\arrow[dr,hook]
\\
& 
Z\otimes Z^\vee
\arrow[rr, crossing over]
&&
X\otimes Z^\vee
\\
Y\otimes Y^*
\arrow[rr,hook]
\arrow[dr,hook]
&&
X\otimes X^*\oplus Y\otimes Y^* = F_2/F_1
\arrow[dr,hook]
\arrow[dl]
\arrow[ur]
\\
&
Z\otimes Y^*
\arrow[from=uu, crossing over]
&&
Z\otimes Z^\vee / F_1 = F_3/F_1
\arrow[ll]
\arrow[uu]
\arrow[from=uull, crossing over, bend left=16]
\end{tikzcd}
$$

By a dual argument, we can analyze $Z^\vee\otimes Z$ by a 3-step filtration with successive quotients $X^*\otimes Y, X^*\otimes X\oplus Y^*\otimes Y,Y^*\otimes X$ given by
$$
G_1= X^*\otimes Y
\subset
G_2 = \ker(Z^\vee\otimes Z \to Y^*\otimes X)
\subset
G_3 = Z^\vee \otimes Z
$$
to construct an evaluation map $\ev_Z:Z^\vee\otimes Z\to \mathbbm{1}$.
One looks at the exact sequence
$$
0
\to
X^*\otimes X\oplus Y^*\otimes Y \cong G_2/G_1
\to
\operatorname{coker}(X^*\otimes Y\to Z^\vee\otimes Z)
\cong G_3/G_1
\to
Y^*\otimes X
\to 
0.
$$
By reasoning dually as before, the morphism $\ev_X\oplus \ev_Y: X^*\otimes X\oplus Y^*\otimes Y\to \mathbbm{1}$ extends to a morphism $G_3/G_1\to \mathbbm{1}$, hence to a morphism $Z^\vee\otimes Z\to \mathbbm{1}$, which we denote by $\ev_Z$.
By similar reasoning, $\ev_Z$ makes the following diagram commute.
$$
\begin{tikzcd}[column sep=4em]
Z^\vee\otimes Y
\arrow[r,"\id_{Z^\vee}\otimes i"]
\arrow[d,"i^\vee \otimes \id_Y"]
&
Z^\vee\otimes Z
\arrow[d,"\ev_Z"]
&
X^*\otimes Z
\arrow[d,"\id_{X^*}\otimes p"]
\arrow[l,"p^\vee\otimes \id_Z", swap]
\\
Y^*\otimes Y
\arrow[r,"\ev_Y"]
&
\mathbbm{1}
&
X^*\otimes X
\arrow[l,"\ev_X"]
\end{tikzcd}
$$
Now one verifies that both $\ev_Z,\coev_Z$ make the following diagram commute.
$$
\begin{tikzcd}
Y
\arrow[rrr,"\coev_Y\otimes \id_Y"]
\arrow[dd,"i"]
\arrow[drr,"\coev_Z\otimes \id_Y"]
&&&
Y\otimes Y^*\otimes Y
\arrow[rrr,"\id_Y\otimes \ev_Y"]
\arrow[dr,"i\otimes \id_{Y^*\otimes Y}"]
&&&
Y
\arrow[dd,"i"]
\\
&&
Z\otimes Z^\vee \otimes Y
\arrow[dr,"\id_{Z\otimes Z^\vee}\otimes i"]
\arrow[rr,"\id_Z\otimes i^\vee\otimes \id_Y"]
&&
Z\otimes Y^* \otimes Y
\arrow[drr,"\id_Z\otimes \ev_Y", near start]
\\
Z
\arrow[rrr,"\coev_Z\otimes \id_Z"]
\arrow[dd,"p"]
\arrow[rrd,"\coev_X\otimes \id_Z"]
&&&
Z\otimes Z^\vee \otimes Z
\arrow[rrr,"\id_Z\otimes \ev_Z"]
\arrow[dr,"p\otimes \id_{Z^\vee\otimes Z}"]
&&&
Z
\arrow[dd,"p"]
\\
&&
X\otimes X^*\otimes Z
\arrow[rr,"\id_X\otimes p^\vee\otimes \id_Z"]
\arrow[dr,"\id_{X\otimes X^*}\otimes p"]
&&
X\otimes Z^\vee\otimes Z
\arrow[rrd,"\id_X\otimes \ev_Z", near start]
\\
X 
\arrow[rrr,"\coev_X\otimes \id_X"]
&&&
X\otimes X^*\otimes X
\arrow[rrr,"\id_X\otimes \ev_X"]
&&&
X
\end{tikzcd}
$$
Looking at the outside rectangle above and the zig-zag $z:Z\to Z$ in the middle, 
since the zig-zags for $X,Y$ are both identities, 
our short exact sequence fits in the following commutative diagram
$$
\begin{tikzcd}
0 
\arrow[r]
&
Y
\arrow[r,"i"]
\arrow[d,equals]
&
Z
\arrow[r,"p"]
\arrow[d,"z"]
&
X
\arrow[r]
\arrow[d,equals]
&
0
\\
0 
\arrow[r]
&
Y
\arrow[r,"i"]
&
Z
\arrow[r,"p"]
&
X
\arrow[r]
&
0.
\end{tikzcd}
$$
Since $\cC$ is abelian, $z$ is an isomorphism by the Short Five Lemma. 
We may thus renormalize one of $\ev_Z,\coev_Z$ so that the zig-zag axiom holds on the nose.
One then uses a similar diagram to check that the zig-zag for $Z^\vee$, which is an idempotent,\footnote{This follows similarly to the proof of Lemma \ref{lem:zig-zag}. 
The zig-zag $z:Z\to Z$ is $\id_Z$, and $z$ is a sub-diagram of $(z')^2$ for the zig-zag $z':Z^\vee\to Z^\vee$.}
is also invertible, and is thus the identity.

(ii)
We have an epimorphism $p: Z\to X$ whose kernel is $Y$. 
We have 
$$
\Hom(Z,X)\cong \Hom(\mathbbm{1},X\otimes Z^*)\cong \Hom(X^*,Z^*),
$$ 
and we write $p^*:X^*\to Z^*$ for the corresponding morphism.\footnote{The morphism
$p^*$ is a monomorphism since for every $T\in \cC$ the map $\Hom(T,X^*)\to \Hom(T,Z^*)$ can be written as a map $\Hom(T\otimes Z,\mathbbm{1})\to \Hom(T\otimes X,\mathbbm{1})$ which is injective since 
$T\otimes Z\to T\otimes X$ is an epimorphism by right exactness of $\otimes$.} Denote by $Y^\vee$ the cokernel of $p^*$. 
We have a morphism $(1\otimes \pi)\circ \coev_Z: \mathbbm{1}\to Z\otimes Y^\vee$, where $\pi: Z^*\to Y^\vee$ is the projection. 

Since $X$ is rigid, it is flat by Lemma \ref{L0}, so ${\rm Tor}^1(X,Y^\vee)=0$. Hence we have a short exact sequence 
$$
0\to Y\otimes Y^\vee\to Z\otimes Y^\vee\to X\otimes Y^\vee\to 0
$$
(cf. \cite[Lem.~4.6]{MR1239507}). 
The morphism $(p\otimes \id_{Y^\vee})\circ (\id_Z\otimes \pi)\circ \coev_Z$ is the morphism $\mathbbm{1}\to X\otimes Y^\vee$ corresponding to the composition $X^*\to Z^*\to Y^\vee$, which is zero by definition of $Y^\vee$. 
Thus the morphism $(\id_Z\otimes \pi)\circ \coev_Z$ lands in $Y\otimes Y^\vee$, and we denote this morphism by $\coev_Y$. 

Similarly, consider the morphism $\ev_Z: Z^*\otimes Z\to \mathbbm{1}$. It pulls back to a morphism \linebreak $\ev_Z\circ (\id_{Z^*}\otimes i): 
Z^*\otimes Y\to \mathbbm{1}$.
Moreover, since the pullback of this morphism to $X^*\otimes Y$ is zero
(as it corresponds to the composition $Y\to Z\to X$) and since we have an exact sequence $X^*\otimes Y\to Z^*\otimes Y\to Y^\vee \otimes Y\to 0$, 
the morphism $\ev_Z\circ (\id_{Z^*}\otimes i)$ 
is the pullback of a unique morphism $Y^\vee\otimes Y\to \mathbbm{1}$, which we denote $\ev_Y$. 
Now diagram chasing shows that the morphisms $\ev_Y,\coev_Y$ equip $Y^\vee$ with the structure of the left dual $Y^*$ of $Y$, as desired. 
%\nn{The relevant diagram is given by}
%$$
%\begin{tikzcd}
%Y
%\arrow[rr,"\coev_Y\otimes \id_Y"]
%\arrow[dr,"\coev_Z\otimes \id_Y"]
%\arrow[dd,"i"]
%&&
%Y\otimes Y^\vee\otimes Y
%\arrow[rr,"\id_Y\otimes \ev_Y"]
%\arrow[dr,"i\otimes \id_{Y^\vee\otimes Y}"]
%&&
%Y
%\arrow[dd,"i"]
%\\
%&
%Z\otimes Z^*\otimes Y
%\arrow[rr,"\id_Z\otimes\pi\otimes \id_Y"]
%\arrow[dr,"\id_{Z\otimes Z^*}\otimes i"]
%&&
%Z\otimes Y^\vee\otimes Y
%\arrow[dr,"\id_{Z}\otimes \ev_Y", near start]
%\\
%Z
%\arrow[rr,"\coev_Z\otimes \id_Z"]
%&&
%Z\otimes Z^*\otimes Z
%\arrow[rr,"\id_Z\otimes \ev_Z"]
%&&
%Z
%\end{tikzcd}
%$$

(iii)
We have a monomorphism $i: Y\to Z$ which by assumption gives rise to an epimorphism $i^*: Z^*\to Y^*$, and we may define $X^\vee$ to be the kernel of $i^*$, so that we have a short exact sequence 
$$0\to X^\vee\to Z^*\to Y^*\to 0.$$ 
Since $Y^*$ is flat, ${\rm Tor}^1(X,Y^*)=0$, so we obtain a short exact sequence 
$$
0\to X\otimes X^\vee\to X\otimes Z^*\to X\otimes Y^*\to 0.
$$
We have a morphism $(p\otimes \id_{Z^*})\circ \coev_Z: \mathbbm{1}\to X\otimes Z^*$ where $p: Z\to X$ is the projection. We have 
$(\id_X\otimes i^*)\circ (p\otimes \id_{Z^*})\circ \coev_Z=0$, 
so the morphism $(p\otimes \id_{Z^*})\circ \coev_Z$ lands in the kernel 
of the morphism $\id_X\otimes i^*: X\otimes Z^*\to X\otimes Y^*$. Thus $(p\otimes \id_{Z^*})\circ \coev_Z$ lands in $X\otimes X^\vee$, and we denote this morphism by $\coev_X$.  

Similarly, we have a morphism 
$\ev_Z\circ (\iota\otimes \id_Z): X^\vee\otimes Z\to \mathbbm{1}$, where $\iota: X^\vee\hookrightarrow Z^*$ is the inclusion, and its pullback to a morphism $X^\vee\otimes Y\to \mathbbm{1}$ corresponds to the morphism $X^\vee\to Y^*$, which is zero by definition of $X^\vee$. But we have an exact sequence
$$
X^\vee\otimes Y\to X^\vee \otimes Z\to X^\vee\otimes X\to 0,
$$
so $\ev_Z\circ (\iota\otimes \id_Z)$ is the pullback of the unique morphism 
$X^\vee\otimes X\to \mathbbm{1}$, which we denote by $\ev_X$. 
Now diagram chasing shows that the morphisms $\ev_X,\coev_X$ equip $X^\vee$ with the structure of the left dual $X^*$ of $X$, as desired. 
%\nn{The relevant diagram is given by}
%$$
%\begin{tikzcd}
%Z
%\arrow[rr,"\coev_Z\otimes \id_Z"]
%\arrow[dd,"p"]
%\arrow[dr,"\coev_X\otimes \id_Z"]
%&&
%Z\otimes Z^*\otimes Z
%\arrow[rr,"\id_Z\otimes \ev_Z"]
%\arrow[dr,"p\otimes \id_{Z^*\otimes Z}"]
%&&
%Z
%\arrow[dd,"p"]
%\\
%&
%X\otimes X^\vee\otimes Z
%\arrow[dr,"\id_{X\otimes X^\vee}\otimes p"]
%\arrow[rr,"\id_X\otimes \iota\otimes \id_Z"]
%&&
%X\otimes Z^*\otimes Z
%\arrow[dr,"\id_X\otimes\ev_Z", near start]
%\\
%X
%\arrow[rr,"\coev_X\otimes \id_X"]
%&&
%X\otimes X^\vee\otimes X
%\arrow[rr,"\id_X\otimes \ev_X"]
%%\arrow[dr,"i\otimes \id_{Y^\vee\otimes Y}"]
%&&
%X
%%\arrow[dd,"i"]
%\end{tikzcd}
%$$
%As in part (i), we conclude that the zig-zag for $X$ is non-zero, and thus we can renormalize $\ev_X,\coev_X$ so that the zig-zag axiom holds on the nose.
\end{proof} 

\begin{rem} 
\label{rem:CannotDropEpiAssumption}
In Lemma \ref{L1}(iii), the assumption that the morphism $i^*:Z^*\to Y^*$ is an epimorphism cannot be dropped. 
For example, let $\cC$ be the category of $\Bbbk[x]$-modules and $X=\Bbbk$, $Y=Z=\Bbbk[x]$. Then the morphism $i: Y\to Z$ is the multiplication map 
$x: \Bbbk[x]\to \Bbbk[x]$, so the dual morphism 
$i^*: Z^*\to Y^*$ is the same, hence not an epimorphism.
In this case $X$ is not rigid, even though $Y,Z$ are. 

See Remark \ref{rem:McRae} below for another example pointed out by Robert McRae in which $\cC$ is artinian.
\end{rem} 

\begin{rem}
A similar result to Lemma \ref{L1}(i) was proved in \cite{MR4315663} for ribbon r-categories, but stated in the language of VOA tensor categories.
However, their result holds in the more general setting of abelian r-categories; see \cite[Thm.~3.13]{2409.14618}.
The interested reader may also extend our Lemma \ref{L1} beyond the braided case 
taking more care about left versus right duals.
(We only used the braiding to check only one kind of duals.)
\end{rem}

\begin{defn} 
We say that $\cC$ is {\bf 2-out-of-3 rigid} if 
in every short exact sequence in $\cC$, if any two of its members are rigid 
then so is the third. 
\end{defn} 

The following corollary follows immediately from Lemma \ref{L1}. 

\begin{cor}\label{C0} Suppose that in $\cC$, 
for any monomorphism $Y\to Z$ between rigid objects, 
the dual morphism $Z^*\to Y^*$ is an epimorphism. 
Then $\cC$ is 2-out-of-3 rigid. 
\end{cor} 

%\begin{proof} This follows from Lemma \ref{L1}. 
%\end{proof} 

\begin{cor}[{cf. \cite[Prop.~A.2, p.~444]{MR1239507}}]\label{C1}
Let $\cC$ be a braided artinian abelian r-category.\footnote{Note that unlike \cite{MR1239507}, we do not assume that $\cC$ is a Frobenius category.} 
Then $\cC$ is 2-out-of-3 rigid.
\end{cor} 

\begin{proof} It is well known that an abelian r-category admits internal homs \cite{MR3134025}, so its tensor product is right exact. 
Since the dualization functor 
$X\to X^\vee$ is an anti-equivalence and $X^\vee\cong X^*$ when $X$ is rigid, the assumption of Lemma \ref{L1}(iii) holds. 
Thus the result follows from Lemma \ref{L1}. 
\end{proof}

\subsection{Recursive filtrations} 

Let $\Bbbk$ be an algebraically closed field. 

\begin{defn}[{\cite[Def.~A.6]{MR1239507}}] 
A {\bf semi-rigid} monoidal category over $\Bbbk$ is an abelian r-category over $\Bbbk$ with objects of finite length and simple $\mathbbm{1}$. 
\end{defn} 

Note that in such a category,
$
\Hom(X,Y^\vee)\cong \Hom(X\otimes Y,\mathbbm{1}),
$ 
so 
$$\dim \Hom(X,Y)\le {\rm length}(X\otimes Y^\vee).$$ 
%where $Y=Z^\vee$. 
Hence $\Hom(X,Y)$ is finite dimensional, so $\cC$ is artinian and $\End(\mathbbm{1})=\Bbbk$.  

\begin{defn} Let $\cC$ be an artinian monoidal category over $\Bbbk$ with the set $\Irr(\cC)$ of isomorphism classes of simple (=irreducible) objects. A {\bf recursive filtration} on $\Irr(\cC)$ relative to a subset $S\subset \Irr(\cC)$ is an increasing exhaustive filtration of $\Irr(\cC)$ by subsets $\Irr_n(\cC)$, $n\in \Bbb Z_{\ge 0}$ such that for every $n\ge 0$ and $X\in \Irr_{n}(\cC)$, there exist $X_1,\cdots,X_m\in \cC$ with composition factors in $S$ for which the composition series of $X_1\otimes\cdots\otimes X_m$ consists of a single copy of $X$ and elements of $\Irr_{n-1}(\cC)$, where by convention, $\Irr_{-1}(\cC):=\emptyset$. 
Note that we may have $m=0$ giving the tensor product of the empty set of objects, which is $\mathbbm{1}$ by definition.
\end{defn} 

\begin{prop}
\label{prop:RecursiveFiltration} 
Suppose $\cC$ is 2-out-of-3 rigid and 
$\Irr(\cC)$ admits a recursive filtration relative to $S$. 
If $S$ 
consists of rigid objects then $\cC$ is rigid. 
\end{prop} 
\begin{proof} It suffices to show that all simple objects in $\cC$ are rigid. 
To this end, it suffices to show that for all $n$, all $X\in \Irr_n(\cC)$ 
are rigid. 
We induct on $n$. 
The base case is trivial.
To pass from $n$ to $n+1$, let $X\in \Irr_{n+1}(\cC)$ and let $X_1,\cdots,X_m\in \cC$ be objects with composition factors in $S$ such that the composition series of $X_1\otimes\cdots\otimes X_m$ consists of a single copy of $X$ and objects in $\Irr_n(\cC)$. 
So there is a 3-step filtration of $X_1\otimes\cdots\otimes X_m$ with terms $Z,X,Y$, where all composition factors of $Z$ and $Y$ are in $\Irr_n(\cC)$. 
Let $T$ be the kernel of the morphism $X_1\otimes\cdots\otimes X_m\to Y$, so we have a short exact sequence 
$$
0\to T\to X_1\otimes\cdots\otimes X_m\to Y\to 0.
$$
Since $X_i$ and (by the induction assumption) $Y,Z$ have rigid composition factors, they are rigid by the 2-out-of-3 property. Hence $X_1\otimes\cdots\otimes X_m$ is rigid. Thus $T$ is rigid by the 2-out-of-3 property. But we also have a short exact sequence 
$$
0\to Z\to T\to X\to 0.
$$ 
Since $Z,T$ are rigid, we conclude by the 2-out-of-3 property that $X$ is rigid, as needed. 
\end{proof} 

Combining Proposition \ref{prop:RecursiveFiltration} with our main result Theorem \ref{thm:FactorialGrowth}, we thus obtain:

\begin{thm}\label{maain} Let $\cC$ be a braided semi-rigid category of moderate growth such that $\Irr(\cC)$ has a recursive filtration 
with respect to a subset $S\subset \Irr(\cC)$ consisting of non-negligible objects. Then $\cC$ is rigid. 
\end{thm} 

\subsection{Application: rigidity of the Kazhdan-Lusztig category at negative rational levels} 
An example of application of Theorem \ref{maain} is a simplification of the proof of rigidity of the category $\mathcal O_\kappa$ in \cite{MR1239507}. Namely, let $\kappa\in \Bbb Q_{<0}$, $\mathfrak{g}$ be a simply laced simple finite dimensional complex Lie algebra,  and $\mathcal O_\kappa$ be the Kazhdan-Lusztig category of $\mathfrak{g}[[t]]$-integrable finitely generated 
representations of the affine Lie algebra $\widehat{\mathfrak{g}}$ at level $\kappa-h^\vee$, where $h^\vee$ is the dual Coxeter number of $\mathfrak{g}$. 
In \cite[Prop~31.2, p.~410]{MR1239507} it is proved that $\mathcal O_\kappa$ is a braided, semi-rigid category, with $\Irr(\cC)=P_+$, the set of dominant integral weights, and it is easy to see that it has moderate growth. 
Let $S:=\lbrace \omega_1,\cdots,\omega_r\rbrace$ 
be the fundamental weights for $\mathfrak{g}$,
let $\rho=\sum_i \omega_i$, and let $\Irr_n(\cC)$ be the set of weights $\lambda\in P_+$ with $(\lambda,\rho)\le n$.
Assume that $\kappa$ is such that the Weyl module $V_\lambda^\kappa$ over $\widehat{\mathfrak{g}}$ is irreducible, and its quantum dimension $d_\lambda(q):=\prod_{\alpha>0} \frac{[(\omega_i+\rho,\alpha)]_q}{[(\rho,\alpha_i)]_q}$ is nonzero at $q=e^{\pi i/\kappa}$ (this rules out an explicit finite set of numerators for $\kappa$, cf.~\cite[p.~413]{MR1239507}). 
Then, as shown in \cite{MR1239507} (see p.~413), $V_{\omega_i}^\kappa\otimes V_{\omega_{i^*}}^\kappa$ contains $\mathbbm{1}$ as a direct summand, so the $V_{\omega_i}^\kappa$ are non-negligible. 
If $\lambda=\sum_i n_i\omega_i$, $n=\sum_i n_i$, and $L_\lambda^\kappa\in \mathcal O_\kappa$ is the simple module with highest weight $\lambda$, then the composition series of 
$\bigotimes_i (V_{\omega_i}^\kappa)^{\otimes n_i}$ consists 
of $L_\lambda^\kappa$ and elements of $\Irr_{n-1}(\cC)$, 
so $\cC$ admits a recursive filtration relative to $S$. 
Thus applying Theorem \ref{maain}, we get that 
$\mathcal O_\kappa$ is rigid, which is \cite[Thm.~32.1, p.~419]{MR1239507}. 
This allows us to avoid quite a few technicalities in the proof of this theorem. 

\begin{rem}  
Conjecture A.1 in \cite[p.~477]{MR1239507} states that in a semi-rigid category, if $X\otimes Y$ is rigid then so are $X,Y$. We note that for braided semi-rigid categories (which is the setting needed in \cite{MR1239507})
this immediately follows from \cite[Cor.~3.5]{MR3134025}. 
\end{rem} 

\begin{rem}
\label{rem:McRae}
We thank Robert McRae for an email exchange providing the following additional example showing the assumption that $i^*$ is an epimorphism in Lemma \ref{L1}(iii) cannot be dropped.
His example comes from Grothendieck-Verdier categories which are not r-categories, and is artinian, whereas the example in Remark \ref{rem:CannotDropEpiAssumption} above is not.

The Weyl modules $V_n=V^{p/q}_{n\omega_1}$ for affine $\mathfrak{sl}_2$ obtained by parabolic induction from the $(n+1)$-dimensional simple $\mathfrak{sl}_2$-module, at level $k = -2 + p/q$ where $p,q$ are relatively prime positive integers, give an exact sequence
$$
0 \to V_{p-2} \to V_1\otimes V_{p-1}\to V_p \to 0.
$$
Here, $V_{p-1}$ is the analogue of the Steinberg representation for quantum $\mathfrak{sl}_2$, or for $sl_2$ in positive characteristic.
In this case, $V_{p-2}$ and $V_1\otimes V_{p-1}$ are rigid, but $V_p$ is not.
Although the Grothendieck-Verdier dual of the inclusion $V_{p-2} \to V_1\otimes V_{p-1}$ is surjective, the rigid dual of $V_{p-2}$ is not the same as its Grothendieck-Verdier dual, so the rigid dual of the inclusion $V_{p-2} \to V_1\otimes V_{p-1}$ cannot be surjective.
\end{rem}

%%%%%%%%%%%%%%%%%%%%%%%%%%%%%%%%%%%%%%%%%%%%%%%%%%%%%%%%%%%%%%%%%%%%%%%%%%%%%%%%%%
\subsection*{Acknowledgements}
This collaboration started at the 2024 Annual Meeting of the Simons Collaboration on Global Categorical Symmetries.
Pavel Etingof is grateful to Yi-Zhi Huang for a discussion at MIT in Fall 2021 which motivated him to think about the problem solved in this paper.
The authors would like to thank
Kevin Coulembier,
Thomas Creutzig,
Andr\'e Henriques,
Corey Jones,
Robert McRae,
Victor Ostrik,
Emily Peters,
Kenichi Shimizu,
and
Harshit Yadav
for helpful conversations.
We thank Victor Ostrik for suggesting that our results should simplify results in \cite{MR1239507}, leading to \S\ref{sec:NonsemisimpleBraidedRCategories}.
Finally, we thank an anonymous referee for supplying the example of an $r$-category that is not weakly rigid in Footnote \ref{footnote:WeaklyRigidNeqRCat}.
Pavel Etingof was supported by NSF grant DMS-2001318 and
David Penneys was supported by NSF grant DMS-2154389.

\bibliographystyle{alpha}
{\footnotesize{

%\bibliography{../../bibliography/bibliography}
}}
\end{document}